\documentclass{article}
\usepackage[utf8]{inputenc}
\usepackage{amsfonts}
\usepackage{amsmath}
\usepackage{amsthm}
\usepackage{graphicx}
\usepackage{float}
\usepackage{xypic}
\usepackage{color}   
\usepackage{hyperref}
\usepackage{tikz-cd}
\usepackage{chngcntr}

\counterwithout{figure}{section}
\newtheorem{lemma}[subsubsection]{Lemma}
\newtheorem*{definition*}{Definition}
\newtheorem*{proposition*}{Proposition}
\newtheorem*{remark*}{Remark}
\newtheorem*{theorem*}{Theorem}
\newtheorem*{lemma*}{Lemma}
\newtheorem*{claim*}{Claim}
\newtheorem{theorem}[subsubsection]{Theorem}
\newtheorem{proposition}[subsubsection]{Proposition}
\newtheorem{definition}[subsubsection]{Definition}
\newtheorem{PandD}[subsubsection]{Proposition-Definition}
\newtheorem{corollary}[subsubsection]{Corollary}
\newtheorem{remark}[subsubsection]{Remark}

\newtheorem{thmx}{Theorem}
\usepackage[textwidth=14cm]{geometry}

\title{Parabolic Components in Cubic Polynomial Slice $Per_1(1)$}
\author{Runze Zhang }
\date{April 2022}

\begin{document}
\maketitle
\begin{abstract}
   We prove that the boundary of every parabolic component in the cubic polynomial slice $Per_1(1)$ is a Jordan curve by  adapting the technique of para-puzzles presented in \cite{Roesch1}. We also give a global description of the connected locus $\mathcal{C}_1$: it is the union of two main parabolic components and the limbs attached on their boundaries.
\end{abstract}

\tableofcontents

\section{Introduction}
Consider the space of unitary cubic polynomials fixing the origin 0:
\[f_{a,\lambda}(z) = z^3+az^2+\lambda z, (a,\lambda)\in\mathbb{C}^2.\]
Let $J_{a,\lambda}$ denote the Julia set of $f_{a,\lambda}$. The connected locus, as the analogue to the Mandelbrot set for the quadratic case, is defined by
\begin{equation*}
\begin{split}
    \mathcal{C} &= \{(a,\lambda)\in \mathbb{C}^2;\,J_{a,\lambda}\text{ is connected}\} \\
       &= \{(a,\lambda)\in \mathbb{C}^2;\text{both critical points of $f_{a,\lambda}$ do not escape to } \infty\}.
\end{split}
\end{equation*}
One of the attempts, proposed by Milnor, to study $\mathcal{C}$ is to restrict ourselves to slices of polynomials by fixing $\lambda$, that is, to consider
\[Per_{1}(\lambda) = \{f_{a,\lambda};\,a\in \mathbb{C}\}\cong\mathbb{C}\]
and its corresponding connected locus $\mathcal{C}_{\lambda} = \{a\in\mathbb{C};\,J_{a,\lambda} \text{ is connected}\}$.

When $|\lambda|\textless 1$, $0$ is always a (super-)attracting fixed point. Define
\[\mathcal{H}^{\lambda} = \{a\in\mathbb{C};\,\text{both critical points of $f_{a,\lambda}$ are attracted by }0\}.\]
This is the union of hyperbolic components of adjacent and capture type, proposed by Milnor in \cite{Milnor}, which are of special interest in this situation. Adjacent means that two critical points belong to the immediate basin of 0; capture means that one in contained in the immediate basin of 0 while the other not (but still attracted by 0). It has been proved in \cite{Faught}, \cite{Roesch1} independently by D. Faught and P. Roesch that the boundary of every component of $\mathcal{H}^{0}$ is a Jordan curve. Combining this with the following result of Petersen-Tan we get that the boundary of every component of $\mathcal{H}^{\lambda}$ is a Jordan curve for $|\lambda|\textless 1$:
\begin{theorem}[\cite{Tanlei}]\label{thm.petertan}
There is a dynamically defined holomorphic motion
\[A: \mathbb{D}\times(\mathbb{C}\setminus\mathcal{H}^0)\longrightarrow\mathbb{C},\,\, (\lambda,a)\mapsto A(\lambda,a) \]
such that $f_{a,0}$ is hybrid conjugated to $f_{A(\lambda,a),\lambda}$ and $A(\lambda,\mathbb{C}\setminus\mathcal{H}^0) = \mathbb{C}\setminus\mathcal{H}^{\lambda}$.
\end{theorem}

In this paper we investigate what happens when $\lambda = 1$ (for simplicity we will omit the multiplier $\lambda$ in the corresponding notations appearing in the rest of the article). Now $0$ becomes a parabolic fixed point for the family $Per_1(1)$, degenerated (which means that it has two attracting axis) if and only if $a = 0$.

For $a\in \mathbb{C}^*$, let $B_a(0)$ be the (unique) parabolic basin and $B_a^*(0)$ the immediate parabolic basin associated to 0, which is defined to be the unique component of $B_a(0)$ containing an attracting petal. One can define the analogue to $\mathcal{H}^{\lambda},\lambda\in\mathbb{D}$:
\[\mathcal{H} = \{a\in \mathbb{C}^*;\,
\text{Both critical points are contained in } B_a(0)\}.\] 
A connected component of $\mathcal{H}$ is called a \textbf{parabolic component}. In this article we study the boundaries of parabolic components and mainly show that
\begin{thmx}\label{thm.A}
The boundary of every parabolic component is a Jordan curve.
\end{thmx}

As an analogue to the adjacent and capture hyperbolic components, we define
\begin{definition*}[adjacent and capture parabolic components]
 $\mathcal{H}$ has a natural decomposition $\mathcal{H} = \bigcup_{n\geq0}\mathcal{H}_n$, where $\mathcal{H}_0$ consists of the parameters $a$ such that both critical points belongs to $B^*_a(0)$ and $\mathcal{H}_n,n\geq 1$ consists of parameters such that one critical point is in $B^*_a(0)$ while the other belongs to $f_a^{-n-1}(B^*_a(0))\setminus f_a^{-n}(B^*_a(0))$. A connected component of $\mathcal{H}_0$ is called a parabolic component of \textbf{adjacent type}\footnote{
We will show that there are only two components of this type. See Corollary \ref{cor.setI}.} (or sometimes called a \textbf{main} parabolic component), that of $\mathcal{H}_n,n\geq1$ is called a parabolic component of \textbf{capture type}.
\end{definition*}

Using the technique of puzzles, we also give a complete description of landing property of external rays for parameters on $\partial\mathcal{H}_0\setminus\{0\}$:

\begin{thmx}\label{thm.B}
Let $a_0\in\partial\mathcal{H}_0\setminus\{0\}$. Then if $a_0$ is renormalizable\footnote{refer to Definition \ref{def.renor}.}, there are exactly two external rays $\mathcal{R}_{\infty}(\eta),\mathcal{R}_{\infty}(\eta')$ landing which bound a copy of Mandelbrot set attached at $a_0$; otherwise there is only one external ray $\mathcal{R}_{\infty}(t)$ landing. Moreover
if $a_0$ is Misiurewicz parabolic\footnote{refer to Definition \ref{def.Misiure}.}, then $t$ satisfies $\exists m,\, 3^mt = 1$.
\end{thmx}

This permits us to define \textbf{wakes} at renormalizable parameters on $\partial\mathcal{H}_0\setminus\{0\}$ to be the region between the two landing rays which contains the Mandelbrot set copy attached at $a_0$. The
\textbf{limbs} attached at $a_0\in\partial\mathcal{H}_0\setminus\{0\}$ is defined to be the intersection of $\mathcal{C}_1$ and the closure of the corresponding wake if $a_0$ is renormalizable and otherwise to be $a_0$. The wakes and limbs at parameter 0 are defined manually, see Definition \ref{def.wake0}.
As a direct corollary of Theorem \ref{thm.B}, we get the following global picture for $\mathcal{C}_1$:
\begin{thmx}\label{thm.C}
The connected locus $\mathcal{C}_1$ can be written as the union of $\mathcal{H}_0$ and the limbs attached on its boundary.
\end{thmx}

\begin{figure}[H] 
\centering 
\includegraphics[width=0.7\textwidth]{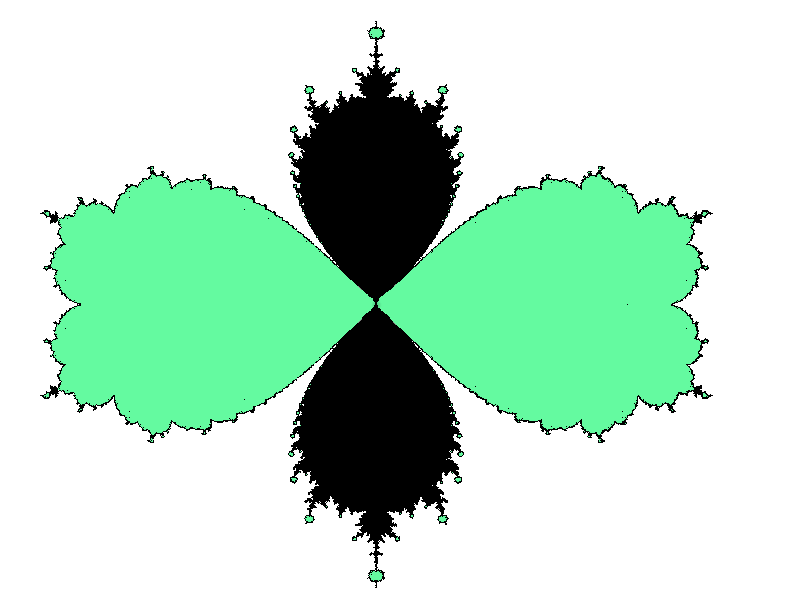} 
\caption{The "Butterfly" $\mathcal{C}_1$. $\mathcal{H}$ is in green. The adjacent type $\mathcal{H}_0$ consists of the two wings of the butterfly. There are two copies of parabolic Mandelbrot sets (in black) attached at the origin, which correspond to parabolic-like maps whose renormalized map has connected Julia set (see L. Lomonaco \cite{lomonaco}, \cite{Lomonaco2012ParameterSF}). The small patches in green attached at the tips are of capture type.} 
\label{Fig.main2} 
\end{figure}

Notice that $\mathcal{C}_1,\mathcal{H}_n,n\geq0$ are symmetric with respect to $x,y-$axis since $\overline{f_a(\overline{z})} = f_{\overline{a}}(z)$, $-f_a(-z)=f_{-a}(z)$ and therefore the dynamics of $f_{\pm a}$ and $f_{\overline{a}}$ can be identified. So essentially, it suffices to study the parameters in $\overline{S}$, where $S = \{x+iy;x\textgreater 0,y\textgreater 0\}$. 

This paper is organized as follows: in Section \ref{sec.parametrisation} we parametrize the parabolic components by locating the free critical value in the basin of the "parabolic model" $z^2+\frac{1}{4}$. A difference from work in the super-attracting case \cite{Roesch1} is that, a priori the free critical point is not marked out, so we need to find out first which critical point appears always on the boundary of the maximal petal. Section \ref{sec.puzzle} is devoted to the construction of parameter puzzles as well as dynamical puzzles. Loosely speaking, parameter puzzles are just "copies" of dynamical puzzles via parametrizations. Conversely they partition the parameter space into pieces on which the dynamical puzzles remain "stable". Moreover one should be careful when dealing with the landing properties of rational parameter rays since $a=0$ might appear on the boundary of connected components of $\mathcal{H}$ while at $a=0$ we lose the stability of repelling petals at the origin. The proof of Theorem \ref{thm.A}, essentially the local connectivity of the boundary of parabolic components $\mathcal{U}$, will be proceeded in Section \ref{sec.loc}. Parameters on $\partial\mathcal{U}$ are divided into three classes: Misiurewicz parabolic type, non-renormalizable type and renormalizable type. The first type corresponds to the "cusps" on which no copy of Mandelbrot set is attached, which do not appear for the boundary of components of $\mathcal{H}^0$, and we will deal with it by puzzles without internal rays (the $\mathcal{Q}_n$ defined at the beginning of subsection \ref{sub.sec.paragraph}). The other two cases left are parallel to those in the super-attracting slice $Per_1(0)$ (see \cite{Roesch1}) and the strategy of the proofs there is adapted. The basic idea is to transfer the "shrinking property" of dynamical puzzles to that of para-puzzles via two key lemmas \ref{lem.homeo.para.dym} and \ref{lem.shishikura}. In Subsection \ref{sub.sec.global} we prove Theorem \ref{thm.B} and Theorem \ref{thm.C}.

\paragraph{Acknowledgements.}
I am grateful to my advisor Pascale Roesch for suggesting me writing this down and for her reviews and comments on this article. I am also grateful to Arnaud Chéritat for his computer program to create pictures of connected locus and Julia sets.

\section{Parametrisation of parabolic components}\label{sec.parametrisation}

We first recall here the following classical result of dynamics associated to parabolic fixed point:
\begin{proposition*}
Let $R:\hat{\mathbb{C}}\longrightarrow\hat{\mathbb{C}}$ be a rational function and $R(0) = 0$, $R'(0) = 1$. Let $B^*$ be an immediate basin of $0$, then there exists
\begin{itemize}
    \item a unique semi-conjugacy $\phi:B^*\longrightarrow\mathbb{C}$ up to translation such that $\phi(R(z)) = \phi(z) +1$.
    \item a unique simply connected open set $\Omega\subset B^*$ whose boundary contains 0 and a critical point, such that it is sent conformally by $\phi$ onto some right-half plane. $\Omega$ is called the \textbf{maximal petal}.
\end{itemize}
\end{proposition*}

\paragraph{Holomorphic dependence of Fatou coordinate.}
From the proof of the existence of Fatou coordinate (cf. \cite{milnor2011dynamics}) one may deduce easily holomorphic dependence of Fatou coordinate for an analytic family in the following sense: let $R_{a}:\mathbb{C}\longrightarrow\mathbb{C}$ be an analytic family (parametrized by $a\in \Lambda\subset \mathbb{C}$, $\Lambda$ open) with $R_a(0) = 0,R'_a(0) = e^{2\pi i\frac{p}{q}}$. Write Taylor expansion near 0:
$R^q_a(z) = z + \omega(a)z^{q+1} + o(z^{q+1})$. Suppose that at some $a_0$, $\omega(a_0)\not = 0$. Then for every attracting (resp. repelling) axis of $R_{a_0}$, there exists
\begin{itemize}
    \item a small neighborhood $D_{a_0}$ of $a_0$; a topological disk $V_{att} = \{x+iy;\,x\textgreater c-b|y|\}$ (resp. $V_{rep} = \{x+iy;\,x\textless c-b|y|\}$), where the constants $b,c$ are positive and do not depend on $a$; a family of attracting (resp. repelling) petals $(P_a)_{a\in B(a_0)}$
    \item a family of Fatou coordinates $(\phi_a)_{a\in D_{a_0}}$ of $R_a$ such that $\phi_a: P_a\longrightarrow V_{att}\,\, (\text{resp. }V_{rep})$ is conformal and $\phi^{-1}_a$ is analytic in $a$.
\end{itemize}
An immediate consequence from this and the $\lambda$-Lemma is that $\phi^{-1}_a\circ{\phi_{a_0}}$ (parametrized by $a\in U_{a_0}$) defines a dynamical holomorphic motion of $\overline{P_{a_0}}$ such that $\phi^{-1}_a\circ{\phi_{a_0}(\overline{P_{a_0}}})=\overline{P_{a}}$.

\paragraph{A choice of inverse branch of critical points.}
One can compute explicitly the two critical points of $f_{a}$: 
\[\Tilde{c}_+(a) = \frac{-a+\sqrt{a^2-3}}{3},\,\,\Tilde{c}_-(a) = \frac{-a-\sqrt{a^2-3}}{3} .\]
where the inverse branch $\sqrt{\cdot}$ is defined on $\mathbb{C}\setminus\mathbb{R}^-$. Hence $\Tilde{c}_\pm(a)$ are continuous for $a\in \mathbb{C}\setminus Z$ where $Z = [-\sqrt{3},\sqrt{3}]\cup i\mathbb{R}$. If we define
\begin{equation}\label{eq.def.crit}
c_+(a)=\left\{
\begin{aligned}
&\Tilde{c}_+(a), a\in\mathbb{H}\setminus (0,\sqrt{3}] \\
&\Tilde{c}_-(a), a\in -\mathbb{H}\setminus [-\sqrt{3},0)  \\
\end{aligned}
\right.\quad
c_-(a) = \left\{
\begin{aligned}
&\Tilde{c}_-(a), a\in\mathbb{H}\setminus (0,\sqrt{3}]\ \\
&\Tilde{c}_+(a), a\in -\mathbb{H}\setminus [-\sqrt{3},0)  \\
\end{aligned}
\right.
\end{equation}
where $\mathbb{H} = \{z = x+iy;\,\,x\textgreater 0\}$ is the right-half plan, then $a\mapsto c_\pm(a)$ can be extended continuously $\mathbb{C}\setminus[-\sqrt{3},\sqrt{3}]$. We will use this choice of inverse branch for the rest of the paper. For $a\in\mathbb{C}\setminus[-\sqrt{3},\sqrt{3}]$, let $v_{\pm}(a) = f_a(c_{\pm}(a))$ be the two corresponding critical values. 
An elementary calculation gives the following asymptotic formulas near $\infty$:
\begin{equation}\label{eq.asymp}
    c_-(a) = -\frac{2a}{3}+O(\frac{1}{a}),\,\,v_-(a) = \frac{4a^3}{27} + O(a).
\end{equation}

\subsection{Basic topological descriptions of $\mathcal{H}_\infty$ and $\mathcal{H}_n$}
Define $\mathcal{H}_\infty = \mathbb{C}\setminus \mathcal{C}_1$ to be the complement of the connected locus. We first prove that:
\paragraph{\uppercase\expandafter{\romannumeral1}. $\mathcal{H}_\infty$ is simply connected.}
 Clearly $\mathcal{H}_\infty$ is open since being attracted by $\infty$ is an open property. Combining the asymptotic formula (\ref{eq.asymp}) of $v_-(a)$ with the following lemma, we get immediately that $\mathcal{H}_{\infty}$ has a connected component containing a neighborhood of $\infty$ on which $c_-(a)$ escapes to $\infty$:
\begin{lemma}\label{lemesti}
There exists $M\textgreater 0$ such that for all $a$ satisfying $|a|\textgreater M$, one has \[\{z;\,\,|z|\textgreater |a|^2\}\subset B_a(\infty),\] 
where $B_a(\infty)$ is the attracting basin of $\infty$.
\end{lemma}
\begin{proof}
Let $\lambda = 2$. Consider
\[
    |f_a(z)|-\lambda|z| \geq |z|^3 -|a|\cdot|z|^2-(1+\lambda)|z|=:g(|z|)
\]
Notice that if $|a|$ is large enough (for example bigger than 10) , $g(x)$ is increasing for $x\geq |a|$. So for $|z|\geq|a|^2$, $g(|z|)\geq g(|a|^2) = |a|^6-|a|^5-(1+\lambda)|a|^2 \textgreater 0$.
Hence by taking iteration we get $f_a^n(z)\to \infty$ when $n\to \infty$.
\end{proof}

The following lemma is a classical result for univalent functions, see for example \cite{Pommerenke}.
\begin{lemma}\label{lemarea}
(Area Theorem). Let $f:\mathbb{C}\setminus\mathbb{D}\longrightarrow \mathbb{C}$ be a univalent holomorphic mapping and suppose $f(z) = z+b_0+\sum_{n=1}^{\infty}b_nz^{-n}$.
Then {$\sum_{n=1}^{\infty}n\cdot|b_n|^2\leq 1$}.
\end{lemma}

For $a\in \mathbb{C}$, let $\phi^{\infty}_a$ denote the Böttcher coordinate at $\infty$ normalized by \nolinebreak{$(\phi^{\infty}_a)'(\infty) = 1$}. 

\begin{proposition}\label{prop.escape}
$\mathcal{H}_{\infty}$ is connected. The mapping $\Phi_{\infty}:{\mathcal{H}}_\infty\longrightarrow \mathbb{C}\setminus \overline{\mathbb{D}}$ defined by $a\mapsto \phi_a^{\infty}(v_-(a))$ is a branched covering of degree 3 ramified at $\infty$.
\end{proposition}

\begin{proof}
Let $\Tilde{\mathcal{H}}_{\infty}$ be any connected component of $\mathcal{H}_{\infty}$. Without loss of generality we may suppose that on $\Tilde{\mathcal{H}}_{\infty}$, $c_-(a)$ escapes to $\infty$. Then one may define $\Phi_{\infty}$ on $\Tilde{\mathcal{H}}_{\infty}$ similarly by $a\mapsto \phi_a^{\infty}(v_-(a))$. Clearly $\Phi_{\infty}$ is holomorphic by holomorphic dependence of Böttcher coordinate.
Let $G_a:\mathbb{C}\longrightarrow \mathbb{R}$ be the Green function of $f_a$ which is continuous on $(a,z)$. Suppose that $(a_n)\subset \Tilde{\mathcal{H}}_\infty$ converges to some $a_0\in\partial \mathcal{C}_1$. Then
\[\lim_{n\to\infty}log\left|\phi^{\infty}_{a_n}(f_{a_n}(c_-(a_n)))\right| = \lim_{n\to \infty}G_{a_n}(f_{a_n}(c_-(a_n))) = G_{a_0}(v_-(a_0)) = 0\]
since $c_-(a_0)$ does not escape to $\infty$. Now suppose $\Tilde{\mathcal{H}}_{\infty}$ is bounded, the above analysis implies that $\Phi(\Tilde{\mathcal{H}}_{\infty})$ is bounded. But on the other hand the above analysis also implies that $\Phi:\Tilde{\mathcal{H}}_{\infty}\longrightarrow\mathbb{C}\setminus\overline{\mathbb{D}}$ is proper, and hence surjective. A contradiction. So $\mathcal{H}_\infty$ is connected, containing a neighborhood of $\infty$.\\
Hence it remains to show that $\Phi_{\infty}(\infty) = \infty$ and its local degree at $\infty$ is 3. To do this, we need an asymptotic analysis on the behavior of ${\phi}_a^{\infty}(v_-(a))$ when $a\to \infty$. 
By Lemma \ref{lemesti}, $\phi_a^{\infty}$ is at least well-defined on $|z|\textgreater |a|^2$ (for $a$ large enough). Write its Taylor expansion:
\[\phi_a^{\infty}(z) = z + b_{a,0} + \sum_{n=0}^{\infty} b_{a,n}z^{-n},\,\,|z|\textgreater |a|^2.\]
Hence $\Tilde{\phi}_a^{\infty}(z) = a^{-2}\phi_a^{\infty}(a^2z)$ satisfies the hypothesis in Lemma \ref{lemarea}, which yields the estimate 
\begin{equation}\label{eqcoeffi}
    \sum_{n=1}^{\infty}\frac{n|b_{a,n}|^2}{|a|^{4n+4}}\leq 1\,\, \Rightarrow \,\,|b_{a,n}|\leq \frac{|a|^{2n+2}}{\sqrt{n}}.
\end{equation}
On the other hand by definition of Böttcher coordinate we have ${\phi}_a^{\infty}\circ f_a= ({\phi}_a^{\infty})^3$:
\[z^3+az^2+z+b_{a,0}+\sum_{n=1}^{\infty}b_{a,n}\frac{1}{(z^3+az^2+z)^n} = \left(z+b_{a,0}+\sum_{n=1}^{\infty}b_{a,n}\frac{1}{z^n}\right)^3.\]
Comparing the constant coefficients one gets $b_{a,0} = 3b_{a,2}+6b_{a,0}b_{a,1}+b^3_{a,0}$. Using (\ref{eqcoeffi}) we obtain $b_{a,0} = O(a^2)$. This together with (\ref{eq.asymp}) and (\ref{eqcoeffi}) yields the desired asympotic estimate when $a\to \infty$:
\[{\phi}_a^{\infty}(v_-(a)) = \frac{4a^3}{27} + O(a^2) +O(a^2)\cdot\sum_{n=1}^{\infty}\frac{1}{\sqrt{n}(a^3+O(a))^n}.\]
\end{proof}

By symmetricity of $\Phi_{\infty}$ with respect to $x,y-$axis one easily gets
\begin{corollary}
The restriction 
\[\Phi_{\infty}:\mathcal{H}_{\infty}\cap{\overline{S}}\longrightarrow \{z;\,\,0 \leq arg(z) \leq \frac{3\pi}{2},|z|\textgreater 1\}\] is a homeomorphism, where $S = \{x+iy;x\textgreater 0,y\textgreater 0\}$.
\end{corollary}

\paragraph{\uppercase\expandafter{\romannumeral2}. Some topological properties for $\mathcal{H}_n$.} We claim that $\mathcal{H}_n$ is open. Indeed, the property of being attracted by 0 is open by holomorphic dependence of Fatou coordinate. Thus any $a\in\mathcal{H}_n$ is J-stable and hence the property of capture or adjacent is preserved under perturbation. Moreover applying the maximal principle to the holomorphic function $f^n_a(c_{\pm}(a))$ is not difficult to show that every component of $\mathcal{H}$ is simply connected (One may need to change the parametrisation to mark out critical points, since under the parameter $a$, $a\mapsto c_{\pm}(a)$ are not well-defined on $\mathbb{C}^*$). The next lemma tells us that the components of $\mathcal{H}_n$ are actually maximal: they are components of $\mathring{\mathcal{C}_1}$.

\begin{lemma}\label{lemboundaryH}
Let $\mathcal{U}$ be a connected component of $\mathcal{H}_n,\,\,n\geq0$. Then $\partial\mathcal{U}\subset\partial \mathcal{C}_1$.
\end{lemma}
\begin{proof}
Suppose the contrary. Then there exists $a_0\in\partial\mathcal{U}$ and a small disk $B(a_0,r)\subset {\mathcal{C}_1}$ on which $\{a\mapsto f^n_a(c_{\pm}(a))\}_n$ are normal families. Indeed, the Böttcher coordinate $\phi^{\infty}_a$ depends holomorphically on $a$ and hence there is a uniform disk at $\infty$ on which $\phi^{\infty}_a$ is well defined for $a\in B(a_0,r)$. Since $B(a_0,r)\subset {\mathcal{C}_1}$, the two critical points never escape to $\infty$, therefore $\{f^n_a(c_{\pm}(a))|_{B(a_0,r)}\}$ are normal. By Mañé-Sad-Sullivan's $J-$stability theorem \cite{MSS}, the Julia set $J_a$ moves homeomorphically and therefore contains no critical points for $a\in B(a_0,r)$. Hence for some $k$, $f^k_{a_0}(c_-(a_0))$ hits a bounded periodic Fatou component $W$. But $W$ can neither be a Siegel disk, nor an attracting basin, nor a parabolic basin different from $B^*_{a_0}(0)$ (since it would be stable under perturbation while $a_0\in\partial\mathcal{U}$). This leads to a contradiction.
\end{proof}

\subsection{Describing the special locus $\mathcal{I}$}
Let $a\in\mathbb{C}^*$. Denote by $\phi_a:B^*_a(0)\longrightarrow\mathbb{C}$ the Fatou coordinate of $f_a$ (not normalized) and $\Omega_a\subset B_a^*(0)$ the maximal petal. Observe that if $a\in[-\sqrt{3},\sqrt{3}]\setminus \{0\}$, then both critical points are on $\partial\Omega_a$ (one may verify by uniqueness of $\phi_a$ that $\Omega_a$ is symmetric with respect to $x-$axis). Define
\[\mathcal{I} = \{a\in\mathbb{C}^*;\, \text{both critical points belongs to } \partial \Omega_a\}.\] 
By definition $\mathcal{I}\subset\mathcal{H}_0$. For $a\in \mathcal{I}$, the image of the two critical points under $\phi_a$ has the same real part. We denote the one whose image has bigger imaginary part by $c^+(a)$ and the other $c^-(a)$. Define $I(a) = \mathfrak{Im}\{\phi_a(c^+(a)) - \phi_a(c^-(a))\}$. It is well-defined due to the uniqueness up to translation of $\phi_a$. Clearly $I\geq 0$.

\begin{remark}
$c^{\pm}(a)$ defined above are not identical with $c_{\pm}(a)$ defined in (\ref{eq.def.crit}). We will only use this notation in this subsection.
\end{remark}

\begin{lemma}\label{lem.rigid}
$I(a) = 0$ if and only $a = \sqrt{3}$. Let $a_1,a_2\in\mathcal{I}$ be in the same connected component of $\mathcal{H}_0$ such that $I(a_1) = I(a_2)\textgreater 0$, then $a_1=a_2$ by the normalisation condition.
\end{lemma}

\begin{proof}
If $a = \sqrt{3}$, the two critical points coincide, therefore $I(a) = 0$. Conversely, $I(a) = 0$ implies that $\phi_a(v^+(a)) - \phi_a(v^-(a)) = 0$, where $v^{\pm}(a) = f_a(c^{\pm}(a))$. Hence $f_a(c^+(a)) = f_a(c^-(a))$ since $\phi_a$ is injective on $\Omega_a$. But this implies that the degree of $f_a$ near $v^+(a)$ is 4, a contradiction.

Choose normalisations of $\phi_{a_i}$, $i=1,2$, such that $\phi_{a_i}(c^{+}(a_i)) = I(a_i)$ and $\phi_{a_i}(c^{-}(a_i)) = 0$ with $\phi_{a_i}(\Omega_{a_i}) = \mathbb{H}$. Set $\Tilde{\Omega}_{a_i} = f_{a_i}(\Omega_{a_i})$ then $f^{-1}_{a_i}(\Tilde{\Omega}_{a_i})$ has three components which are topological disks: ${\Omega}_{a_i},{\Omega}^+_{a_i},{\Omega}^-_{a_i}$ and $\partial{\Omega}_{a_i}\cap\partial{\Omega}^+_{a_i} = c^+(a_i)$, $\partial{\Omega}_{a_i}\cap\partial{\Omega}^-_{a_i} = c^-(a_i)$ (see Figure \ref{fig.illus}). We set $K_i =\overline{{\Omega_{a_i}}\cup{\Omega^+_{a_i}}\cup{\Omega^-_{a_i}}}$. Therefore $f_{a_1}$ is conjugate to $f_{a_2}$ on $K_1$ by $\varphi = \phi^{-1}_{a_2}\circ\phi_{a_1}$ where $\varphi$ is defined by choosing the inverse branch of $\phi_{a_2}|_{\Omega_{a_2},\Omega^{\pm}_{a_2}}$ respectively. Now $B_{a_i}^*(0)\setminus K_i$ has three connected components denoted by $D^i_1,D^i_2,D^i_3$ in the cyclic order. Set $\Omega^1_{a_i} = f^{-1}_{a_i}(\Omega_{a_i})$, then $\Omega^1_{a_i}\setminus K_i$ has three connected components $\Tilde{C}^i_\epsilon\subset D^i_\epsilon$, $\epsilon = 1,2,3$, on which $f_{a_1}$ is conjugate to $f_{a_2}$ by $\phi^{-1}_{a_2}\circ\phi_{a_1}$ by choosing $\phi^{-1}_{a_2}(\{z;\,0\textless\mathfrak{Re}\textless1\}) = \Tilde{C}^2_\epsilon$. One checks easily that the conjugating mapping defined on $\Tilde{C}^i_\epsilon$ coincides with $\varphi$ on the boundary, therefore by Morera's Theorem $\varphi$ extends holomorphically to $\Omega^1_{a_1}$ conjugating $f_{a_1}$ to $f_{a_2}$. For $n\geq 1$, define $\Omega^{n+1}_{a_i} = f^{-1}_{a_i}(\Omega^{n}_{a_i})$ and the components $C^i_{\epsilon_0...\epsilon_n}$ of $\Omega^{n+1}_{a_i}\setminus\overline{\Omega^{n}_{a_i}}$ where $\epsilon_k\in \{1,2,3\}, k =0,...,n-1$ is determined by $f^k_{a_i}(C^i_{\epsilon_0...\epsilon_{n-1}}) \subset D^i_{\epsilon_k}$. Define conjugating mappings on these components via Fatou coordinates by assigning $C^1_{\epsilon_0...\epsilon_{n-1}}$ to $C^2_{\epsilon_0...\epsilon_{n-1}}$ and one may prove by induction that these mappings coincide with $\varphi$ on $\partial\Omega^n_{a_1}$. Thus again by Morera's Theorem $\varphi$ extends holomorphically to $\Omega^{n+1}_{a_1}$ and hence to $B^*_{a_1}(0)$ by induction on $n$.

Since $f_{a_i}$ is geometrically finite, $\partial B^*_{a_i}(0)$ is a Jordan curve by a result of Tan-Yin \cite{TaYi}. Therefore $\varphi$ continuously extends to $\partial B^*_{a_1}(0)$. On the other hand, since $a_1,a_2$ are in the same component $\mathcal{U}_0$ of $\mathcal{H}_0$, the dynamical holomorphic motion (parametrized by $a\in\mathcal{U}_0$) $\psi_a :=(\phi^{\infty}_a)^{-1}\circ{\phi^{\infty}_{a_1}}$ of $\mathbb{C}\setminus \overline{B^*_{a_1}(0)}$ with $\psi_a(\mathbb{C}\setminus \overline{B^*_{a_1}(0)})= \mathbb{C}\setminus \overline{B^*_{a}(0)}$ can be extended to a holomorphic motion of $\mathbb{C}$ by Slodkowski's Theorem. Now $\psi_{a_2},\varphi$ both conjugate $f_{a_1}$ to $f_{a_2}$ on $\partial B^*_{a_1}(0)$ and fix 0 since they are dynamically defined. Apply the following lemma to $f_{a_1},f_{a_2}$ and $\psi_{a_2},\varphi$, we get that $\psi_{a_2} = \varphi$ on $\partial B^*_{a_1}(0)$.

\begin{lemma}[\cite{Tanlei}]
Let $f_1,f_2: \mathbb{S}^1\longrightarrow\mathbb{S}^1$ be weakly expanding map of degree $d\geq 2$ ($f$ is weakly expanding means that for any closed segment $I\subset \mathbb{S}^1$, there exists $N$ such that $f^N(I) = \mathbb{S}^1$). Let $\alpha_1,\alpha_2$ be a fixed point of  $f_1,f_2$ respectively. Then there exists a unique orientation-preserving homeomorphism $h: \mathbb{S}^1\longrightarrow\mathbb{S}^1$ such that $h \circ{f_1} = f_2\circ{h}$ with $h(\alpha_1) = \alpha_2$.
\end{lemma}

Hence the mapping $\psi$ defined by $\psi|_{\mathbb{C}\setminus B^*_{a_1}(0)} = \psi_{a_2}$ and $\psi|_{\overline{B^*_{a_1}(0)}} = \varphi$ is continuous. By Rickman's Lemma, $\psi$ is quasiconformal. In fact it is conformal: its Beltrami coefficient vanishes almost everywhere since $\partial B^*_{a_1}(0)$ has zero measure (see for example \cite{DH}). Hence $a_1 = a_2$.
\end{proof}

\begin{figure}[H] 
\centering 
\includegraphics[width=0.6\textwidth]{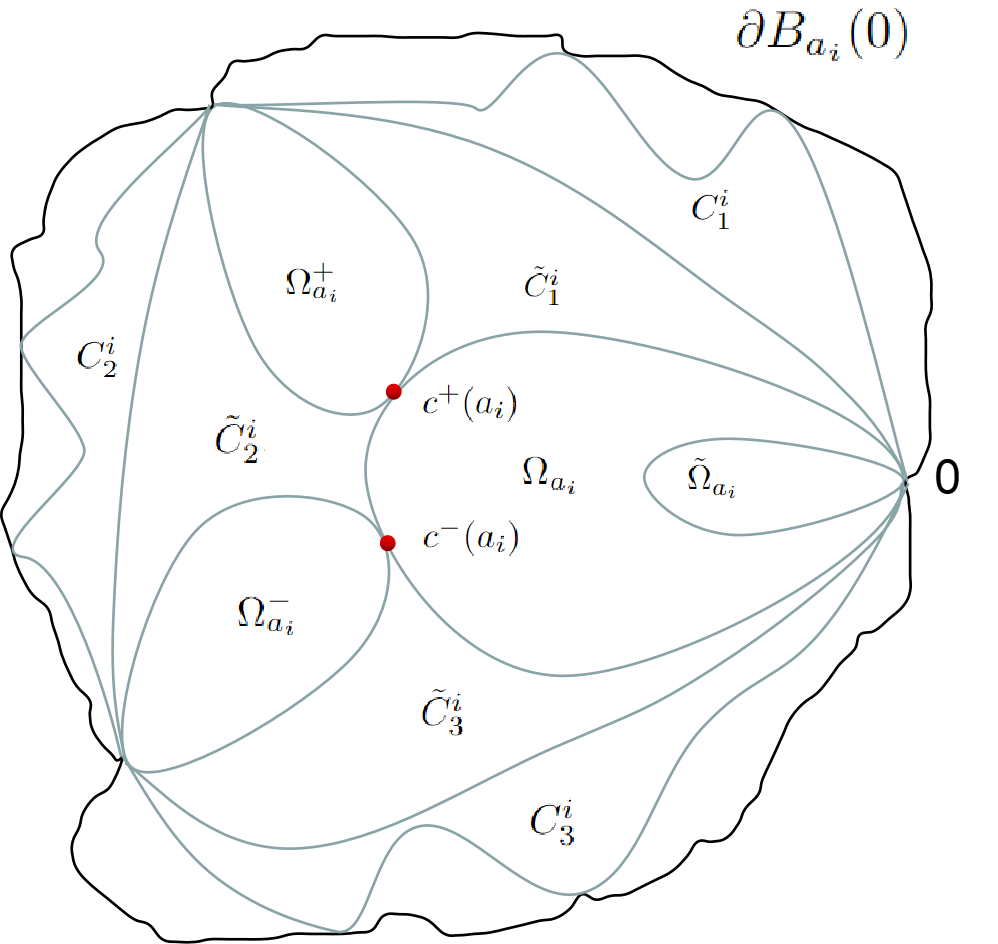} 
\caption{Illustration of the proof of Lemma \ref{lem.rigid}} 
\label{fig.illus} 
\end{figure}

\begin{lemma}
Let $\mathcal{U}_0$ be a connected component of $\mathcal{H}_0$ such that $\mathcal{U}_0\cap\mathcal{I}\not = \emptyset$. Then $I:(\mathcal{I}\setminus\{\pm\sqrt{3}\})\cap\mathcal{U}_0\longrightarrow(0,+\infty)$ is a homeomorphism. Moreover $\lim\limits_{t\to \infty}I^{-1}(t) = 0$, and $\lim\limits_{t\to 0}I^{-1}(t)$ is either $\pm\sqrt{3}$ or $0$.  
\end{lemma}
\begin{proof}
By Lemma \ref{lem.rigid}, $I$ is injective on $\mathcal{U}_0\cap\mathcal{I}$, so it suffices to prove that $I$ is also surjective. Pick $a_0\in\mathcal{U}_0\cap\mathcal{I}$ such that $t_0 = I(a_0)\textgreater 0$. For $t\in(0,+\infty)$, define $\tau_t:\mathbb{C}\longrightarrow\mathbb{C}$ by $(x,y)\mapsto(x,\frac{t}{t_0}y)$. Define a $f_{a_0}-$invariant Beltrami differential $\mu_t = \phi_{a_0}^*0$ on $B^*_{a_0}(0)$, $\mu_t = 0$ elsewhere. Let $\psi_t$ integrate $\mu_t$ with normalisation $\psi_t(0) = 0$ and $\psi_t(z) = z+o(1)$ near $\infty$. Then $f_{a_t} = \psi_t\circ f_{a_0}\circ \psi^{-1}_t \in Per_1(1)$ since the dynamic at the parabolic fixed point 0 is preserved by $\psi_t$. For the reason of consistency of notations, we denote $a_{t_0} = a_0$. Notice that $\tau_t\circ{\phi_{a_{t_0}}}\circ{\psi^{-1}_t}$ is a Fatou coordinate of $f_{a_t}$, on checks easily that $I(a_t) = t$. So $I$ is surjective.

If $\mathcal{U}_0$ is the component containing $(0,\sqrt{3}]$ or $[-\sqrt{3},0)$, then we are done. So suppose $\mathcal{U}_0$ is another component and hence on which $a\mapsto c_{\pm}(a)$ is well-defined and holomorphic on $a$. Let $b_{\infty},b_0$ be an accumulation point of $I^{-1}(0,+\infty)$ when $t\to \infty,0$ respectively. If $b_{\infty}\not = 0$, let $(a_n)\subset I^{-1}(0,+\infty)$ converge to $b_{\infty}$ as $n\to\infty$, then $\phi^{-1}_{b_{\infty}}(\phi_{a_n}(c_-(a_n)))$ converges to $c_-(b_\infty)$ by holomorphic dependence of Fatou coordinate. But $\phi_{a_n}(c_-(a_n))\to\infty$ as $n\to\infty$, which implies that $c_-(b_{\infty}) = 0$, a contradiction.

If $b_{0}\not = 0$, let $(a_n)\subset I^{-1}(0,+\infty)$ converge to $b_{0}$ as $n\to\infty$, then by holomorphic dependence of Fatou coordinate, $c_-(b_0)$ is attracted by 0, therefore $b_0\in\mathcal{U}_0$. Hence the function $|\phi_a(c_+(a))-\phi_a(c_-(a))|$ is well-defined in a neighborhood of $b_0$ and is continuous. Thus $\phi_{b_0}(c_-(b_0)) = \phi_{b_0}(c_+(b_0))$, which implies $c_-(b_0) = c_+(b_0)$, i.e. $b_0 = \pm\sqrt{3}$.
\end{proof}

We get immediately
\begin{corollary}
Let $\mathcal{U}_0$ be as in the above Lemma. If $([-\sqrt{3},\sqrt{3}]\setminus\{0\})\cap\mathcal{U}_0\not=\emptyset$, then $\mathcal{I}\cap\mathcal{U}_0 = (0,\sqrt{3}]$ or $[-\sqrt{3},0)$. Otherwise $(\mathcal{I}\cap\mathcal{U}_0)\cup\{0\}$ is a simple closed curve.
\end{corollary}

\subsection{Parametrisation of $\mathcal{H}_n$, $n\geq 0$}
The main result in this subsection are Proposition \ref{prop.para.Hn} and \ref{PROP_para0}.

In order to give an explicit parametrisation, we need first to distinguish the relative position of the two critical points. Define
\[\mathcal{A}_+ = \{a\in \mathbb{C}^*;\,\,c_+(a)\in\partial\Omega_a,c_-(a)\not \in \overline{\Omega_a}\}.\]
\[\mathcal{A}_- = \{a\in \mathbb{C}^*;\,\,c_-(a)\in\partial\Omega_a,c_+(a)\not \in\overline{\Omega_a}\}.\]
By holomorphic dependence of Fatou coordinate, both $\mathcal{A}_{\pm}$ are open. By definition, $\partial\mathcal{A}_{+}\cup\partial\mathcal{A}_{-} = \mathcal{I}\cup\{0\}$. 

Since $c_+(a)$ will never escape to $\infty$ by Proposition \ref{prop.escape}, it is reasonable to believe that it appears always on the boundary of the maximal petal. In fact, we will show that $\mathcal{A}_- = \emptyset$ (Corollary \ref{cor.setI}).

\paragraph{Parametrization by locating the free critical value}
A direct calculation shows that the equation $v_-(a) = c_+(a)$ has exactly four solutions $\pm s_0,\pm\overline{s_0}$ with $s_0\in S$. Therefore $\mathcal{H}_0\cap\mathcal{A}_+$ is non empty. For $a_0\in \mathcal{H}_0\cap\mathcal{A}_+$, let $\phi_{a_0}:\Omega^0_{a_0} :=\Omega_{a_0}\longrightarrow \mathbb{H}$ be the normalized Fatou coordinate with $\phi_{a_0}(c_+(a_0)) = 0$. On the other hand, set $P_{\frac{1}{4}}(z) := z^2+\frac{1}{4}$, denote by $\Omega^0$ the maximal petal of $P_{\frac{1}{4}}$, $B$ its parabolic basin and $\phi:\Omega^0\longrightarrow \mathbb{H}$ the normalized Fatou coordinate with $\phi(0) = 0$. For $n\geq 1$, let $\Omega_{a_0}^n$ (resp. $\Omega^n$) be the connected component of $f_{a_0}^{-1}(\Omega_{a_0}^{n-1})$ (resp. $P_{\frac{1}{4}}^{-1}(\Omega^{n-1})$) containing $\Omega_{a_0}^{n-1}$ (resp. $\Omega^{n-1}$). Define $h_a:\Omega^0_a\longrightarrow\Omega^0$ by $h_{a_0} = \phi^{-1}\circ\phi_{a_0}$. We can lift $h_{a_0}$ by the following commutative diagram until $\Omega_{a_0}^{n_{a_0}}$ contains $v_-(a_0)$ for the first time, since before that $h_{a_0}(c_+({a_0})) = 0,h_{a_0}(v_+(a_0)) = P_{\frac{1}{4}}(0)$:
\begin{equation}\label{commu1}
\begin{tikzcd}
\Omega_a^{n_{a_0}} \ar[r]{}{f_{a_0}}\ar[d]{}{h_{a_0}} & \Omega_a^{{n_{a_0}}-1} \ar[d]{}{h_{a_0}} \ar[r]{}{f_{a_0}} & ... \ar[d]{}{h_{a_0}}  \ar[r]{}{f_{a_0}} & \Omega_{a_0}^1 \ar[d]{}{h_{a_0}} \ar[r]{}{f_{a_0}} & \Omega_{a_0}^0 \ar[d]{}{h_{a_0}}\\
\Omega^{n} \ar[r]{}{P_{\frac{1}{4}}} & \Omega^{n-1}  \ar[r]{}{P_{\frac{1}{4}}} & ...  \ar[r]{}{P_{\frac{1}{4}}} & \Omega^1  \ar[r]{}{P_{\frac{1}{4}}} & \Omega^0
\end{tikzcd}
\end{equation}
Define $\Phi_+:\mathcal{H}_0\cap{\mathcal{A}_+}\longrightarrow B$ by $\Phi_+(a) = h_a(v_-(a))$.
It would be a little subtle to see that $\Phi_+$ is holomorphic: regard $h_a(z)$ as a function of two variables, then it is defined on \[D = \{(a,z);\,\,a\in\mathcal{H}_0\cap{\mathcal{A}_+}, z\in\Omega^{n_a}_a\}.\]
If $v_-(a_0)\in\Omega^{n_{a_0}}\setminus\overline{\Omega^{n_{a_0-1}}_{a_0}}$, then by holomorphic dependance of Fatou coordinate we have $n_a = n_{a_0}$ for $a$ close to $a_0$, hence $(a_0,v_-(a_0))\in\mathring D$. If $v_-(a_0)\in\partial\Omega^{n_a-1}_{a_0}$, then there is a chance that $v_-(a_0)\in\Omega^{n_{a_0}-1}_{a}$ for $a$ near $a_0$ (in this case $n_a = n_{a_0} - 1$) so that $(a_0,v_-(a_0))$ is no longer in $\mathring D$. However, we may take a petal $U^0_{a_0}\subset \Omega^0_{a_0}$ such that its $n_{a_0}-$th pull-back $U^{n_{a_0}-1}_{a_0}$ contains $v_-(a_0)$ on its boundary, and therefore it can be pulled back once again. So $h_a(z)$ is defined on a open neighborhood of $(a_0,v_-({a_0}))$ and by chain rule $\Phi_+$ is holomorphic.

On can construct $\Phi_-$ on $\mathcal{H}_0\cap\mathcal{A}_-$ similarly by locating the position of $v_+(a)$.

Notice that if $a\not \in\mathcal{H}_0$, then diagrame (\ref{commu1}) can be continued forever and $h_a$ can be extended to $B^*_a(0)$. In \cite{Nakane} Nakane gives a parametrisation of $\mathcal{H}_n,n\geq 1$ by (admitting that $c_+(a)$ is always contained in the immediate basin)
\begin{proposition}\label{prop.para.Hn}
Let $\mathcal{U}$ be a connected component of $\mathcal{H}_n$ ($n\geq 1$). Then $\Phi_{\mathcal{U}}:\mathcal{U}\longrightarrow B$ defined by $a\mapsto h_a(f^{n+1}(c_-(a)))$ is a homeomorphism.
\end{proposition}

\begin{corollary}
There are only finitely many components of $\mathcal{H}_n$ ($n\geq 1$).
\end{corollary}
\begin{proof}
From the above proposition we see that for every component of $\mathcal{H}_n$ there is a unique parameter $a_0$ such that $\Phi_{\mathcal{U}}(a_0) = 0$. This leads to $f^{n+1}_{a_0}(c_-(a_0)) = c_+(a_0)$, which is a non -degenerated algebraic equation, hence only finitely many solutions.
\end{proof}
However there he does not justify why the critical point $c_+(a)$ should always appear on $\partial\Omega_a$. Here we first answer positively to this question by Corollary \ref{cor.setI} and then our job reduces to parametrize $\mathcal{H}_0$, which will be done in Proposition \ref{PROP_para0}.

\begin{lemma}\label{lem.proper}
Let $\Tilde{\Omega} = P_{\frac{1}{4}}(\Omega^0)$. Let $\mathcal{U}_0$ be a connected component of $\mathcal{H}_0$. Consider the restriction $\Phi_{\pm}:\mathcal{U}_0\cap\mathcal{A}_{\pm}\longrightarrow B$.
\begin{itemize}
    \item If $\mathcal{U}_0\cap\mathcal{I}\not = \emptyset$, then $\Phi_{\pm}:W\longrightarrow \Tilde{\Omega}\text{ or }B\setminus\overline{\Tilde{\Omega}}$ is proper, where  $W$ is a connected component of $\Phi^{-1}_{\pm}(\Tilde{\Omega})$ or of $\Phi^{-1}_{\pm}(B\setminus\overline{\Tilde{\Omega}})$.
    \item If $\mathcal{U}_0\cap\mathcal{I} = \emptyset$, then $\Phi_{\pm}:\mathcal{U}_0\cap\mathcal{A}_{\pm}(=\mathcal{U}_0)\longrightarrow B$ is proper.
\end{itemize}
\end{lemma}

\begin{proof}
We only prove the case when $W$ is a component of $\Phi_+(B\setminus\overline{\Tilde{\Omega}})$ and the other cases are similar. Let $(a_n)\subset W$ be a sequence tending to some $a_0\in\partial W$. Since $W\subset \mathcal{H}_0\cap\mathcal{A}_+$, either $a_0\in \mathcal{H}_0\cap\mathcal{A}_+$ or $a_0\in \partial(\mathcal{H}_0\cap\mathcal{A}_+)$. If $a_0\in \mathcal{H}_0\cap\mathcal{A}_+$, then $\Phi_{+}(a_n)$ tends to $\partial\Tilde{\Omega}$ by continuity of $\Phi_{+}$. So let $a_0\in \partial(\mathcal{H}_0\cap\mathcal{A}_+)\subset(\partial\mathcal{H}_0\cup\partial\mathcal{A}_+) = \partial \mathcal{H}_0\cup\mathcal{I}$. If $a_0\in \mathcal{I}$, then $\Phi_{+}(a_n)$ tends to $\partial\Tilde{\Omega}$ by continuity of Fatou coordinate. 

So suppose $a_0\in\partial \mathcal{H}_0$. We want to prove that $\Phi_+(a_n)$ converges to $\partial B$. Suppose the contrary: up to taking a subsequence, $\Phi_+(a_n)$ converges to some $b\in B$. Then there exists some $k\geq 0$ not depending on $n$ such that $b\in \Omega^k$ and $h_{a_n}$ is invertible on $\Omega^k_{a_n}$. Consider the normal family $h^{-1}_{a_n}: \Omega^k\longrightarrow \Omega^{k}_{a_n}$. Up to taking a subsequence we may suppose that it converges uniformly on compact sets to some $\psi:\Omega^k\longrightarrow \mathbb{C}$. We claim that $\psi(\Omega^k)$ will not intersect the Julia set $J_{a_0}$ of $f_{a_0}$. Indeed, if $\psi(\Omega^k)\cap J_{a_0}\not = \emptyset$, then their intersection contains at least a repelling periodic point $x_0$ of $f_{a_0}$. Then near $a_0$ there exists a holomorphic motion of $x_0$, therefore for $a$ close enough to $a_0$, $h^{-1}_{a_n}(\Omega^k)$ contains a repelling periodic point, a contradiction. Hence if $a_0 \not = 0$, then $\psi(\Omega^k)\subset B_{a_0}^*(0)$ since $f_{a_n}(c_+(a_n)) = h^{-1}_{a_n}(\frac{1}{4}) \to \psi(\frac{1}{4})$, but 
\[\psi(b) = \lim_{n\to \infty}h^{-1}_{a_n}(b) = \lim_{n\to \infty}f_{a_n}(c_-(a_n)) =  v_-(a_0)\in B_{a_0}^*(0).\]
A contradiction since $a_0\not\in \mathcal{H}_0$. A similar argument as above works for $a_0 = 0$, The only difference is that we replace $B_{a_0}^*(0)$ by $B_{+}^*(0)$ or $B_{-}^*(0)$, where $B_{\pm}^*(0)$ is the immediate basin of $f_0$ contained in the upper-half (resp.lower-half) plane.
\end{proof}

\begin{corollary}\label{cor.setI}
$\mathcal{H}_0$ has exactly two connected components $\mathcal{U}_0,-\mathcal{U}_0$ symmetric with respect to $y-$axis, containing $(0,\sqrt{3}],[-\sqrt{3},0)$ respectively. In particular, $\mathcal{A}_- = \emptyset$, $\mathcal{I} = [-\sqrt{3},\sqrt{3}]\setminus\{0\}$. 
\end{corollary}
\begin{proof}
Suppose there is a component $\mathcal{U}'_0$ not intersecting $[-\sqrt{3},\sqrt{3}]\setminus\{0\}$. If $\mathcal{I}\cap\mathcal{U}'_0= \emptyset$, then by the second point of Lemma \ref{lem.proper}, there exists $a_0\in\mathcal{U}'_0$ sent to $\frac{1}{4}$ by $\Phi_+$ or $\Phi_-$, which implies that $a_0 = \pm\sqrt{3}$, a contradiction. Hence $\mathcal{I}\cap\mathcal{U}'_0\not= \emptyset$ and by Corollary \ref{cor.setI} and Lemma \ref{lemboundaryH}, $(\mathcal{I}\cap\mathcal{U}'_0)\cup\{0\}$ bounds a connected component $W$ (which is simply connected) of $\mathcal{A}_{+}$ or of $\mathcal{A}_{-}$. By the first point of Lemma \ref{lem.proper}, $\Phi_+$ or $\Phi_-$ induces a proper map from $W$ to $\Tilde{\Omega}$ and hence can be extended to $\partial W$. Therefore $0$ has to be sent to $\frac{1}{2}$ (which is the parabolic fixed point of $z^2+\frac{1}{4}$) by the extended map of $\Phi_+$ or $\Phi_-$. However there will not be any point on $\partial W$ that is sent to $\frac{1}{4}$ since if $\Phi_{\pm}(a) = \frac{1}{4}$, then the two critical points of $f_a$ coincide, i.e. $a = \pm\sqrt{3}$.
\end{proof}

\paragraph{Notation.} By the above corollary, $c_+(a)$ is always on the maximal petal $\Omega_a$ so there is no need to define $\Phi_-$ and for simplicity we will replace $\Phi_+$ by $\Phi$. In the sequel, we may write $\mathcal{H}_0 = \mathcal{U}_0\cup(-\mathcal{U}_0)$, where $\mathcal{U}_0$ is the connected component containing $(0,\sqrt{3}]$.

\begin{definition}\label{def.Misiure}
A parameter $a\in\mathcal{C}_1\setminus[-\sqrt{3},\sqrt{3}]$ is of \textbf{Misiurewicz} type if $v_-(a)$ is a repelling (pre-)periodic point; $a\in\mathcal{C}_1\setminus[-\sqrt{3},\sqrt{3}]$ is is of \textbf{Misiurewicz parabolic} type if $v_-(a)$ is in the inverse orbit of 0.
\end{definition}

Now we give a more specific description of $\Phi:\mathcal{U}_0\setminus\mathcal{I}\longrightarrow B$ (see Figure \ref{fig.main.parabolic}):
\begin{proposition}\label{PROP_para0}
$\Phi^{-1}(\Tilde{\Omega})\cap\mathcal{U}_0$ has only one connected component $W$ symmetric with respect to the real axis and containing the open interval $(\sqrt{3},2)$. Its boundary $\partial W = \gamma\cup\tau(\gamma)$ ($\tau(z) = \overline{z}$) where $\gamma \subset \overline{S}\cap\mathcal{U}_0$ is a simple arc joining $\sqrt{3}$ and $2$. $\Tilde{\mathcal{U}}_0 := \mathcal{U}_0\setminus(\mathcal{I}\cup \overline{W})$ has two connected components symmetric with respect to the real axis. Moreover, $\Phi|_{\Tilde{\mathcal{U}}_0\cap S}$ and $\Phi|_W$ are homeomorphisms mapping onto $B\setminus\overline{\Tilde{\Omega}}$ and $\Tilde{\Omega}$ respectively.
\end{proposition}
\begin{proof}
First of all one can verify easily by symmetricity of Fatou coordinate and Julia sets of $f_a$ when $a\in \mathcal{R}$ that $(\sqrt{3},{2})\subset \Phi^{-1}(\Tilde{\Omega})$ while $(2,+\infty)\in \mathcal{H}_{\infty}$. 

Let $W$ be the connected component of $\Phi^{-1}(\Tilde{\Omega})$ containing $(\sqrt{3},{2})$. When $a\in \Phi^{-1}(\Tilde{\Omega})$, $f^{-1}_a(\overline{\Tilde{\Omega}_a})$ (where $\Tilde{\Omega}_a = f_a(\Omega_a)$) is an "eight": its interior consists of two simply connected components, one is $\Omega_a$, the other $\Omega^-_a$ attached $\Omega_a$ at $c_+(a)$ contains $c_-(a)$ and is mapped 2 to 1 onto $\Tilde{\Omega_a}$. So if $a_1,a_2\in \Phi^{-1}(\Tilde{\Omega})\cap\mathcal{U}_0$ with $\Phi(a_1) = \Phi(a_2)$, then it is easy to construct a conjugcy map $\varphi$ on this "eight" via Fatou coordinates. By applying the same strategy as in Lemma \ref{lem.rigid}, one may prove that $\varphi$ admits an conformal extension on $\mathbb{C}$ and hence prove the injectivity of $\Phi$ on $\Phi^{-1}(\Tilde{\Omega})$. In particular $W = \Phi^{-1}(\Tilde{\Omega})$. Similarly one may prove that $\Phi$ is injective on $\Phi^{-1}(\partial\Tilde{\Omega})\cap\mathcal{U}_0$. In fact, this is equivalent to the injectivity of $I(a) = \phi_a(c_+(a))-\phi_a(c_-(a))$ on $\Phi^{-1}(\partial\Tilde{\Omega})\cap\mathcal{U}_0$. Therefore $\Tilde{\gamma} = \Phi^{-1}(\partial\Tilde{\Omega})\cap S$ is parametrized by $I^{-1}:(0,+\infty)\longrightarrow \Tilde{\gamma}$. Let $b_0,b_\infty$ be an accumulation point of $I^{-1}(t)$ when $t$ tends to 0 and $\infty$ respectively. By holomorphic dependence of Fatou coordinate, $b_0 =0$ or $\sqrt{3}$, $b_\infty = 0$ or $2$ (notice that $a=2$ comes from the solution of $f_-(c_-(a)) = 0$). So there are four possibilities: $b_0 =0, b_{\infty} = 0$; $b_0 =\sqrt{3}, b_{\infty} = 0$; $b_0 =0, b_{\infty} = 2$; $b_0 =\sqrt{3}, b_{\infty} = 2$. The first can not happen for otherwise $\overline{\Tilde{\gamma}}$ surrounds a region not intersecting $W$ that is mapped to $\Tilde{\Omega}$ under $\Phi$. The the second is not true for a similar reason. If we are in the third case, then $\partial W = \overline{\Tilde{\gamma}}\cup\tau(\overline{\Tilde{\gamma}})\cup(0,\sqrt{3})$. Thus $\Phi^{-1}:\Tilde{\Omega}\longrightarrow W$ admits a continuous extension to the boundary by Carathéodory Theorem. But $\Phi^{-1}(z)$ accumulates to both $\sqrt{3},0$ when $z\to \frac{1}{4}$, a contradiction. So $\gamma = \Tilde{\gamma}$ is an arc joining $\sqrt{3},2$ and $\partial W = \gamma\cup\tau(\gamma)$.

So it remains to prove that $\Phi|_{\Tilde{\mathcal{U}}_0\cap S}$ is injective. This is easy to see since the equation $v_-(a) = c_+(a)$ has exactly four solutions $\pm s_0,\pm\overline{s_0}$ with $s_0\in S$. The injectivity follows easily by the properness of $\Phi$.
\end{proof}

\begin{figure}[H] 
\centering 
\includegraphics[width=0.8\textwidth]{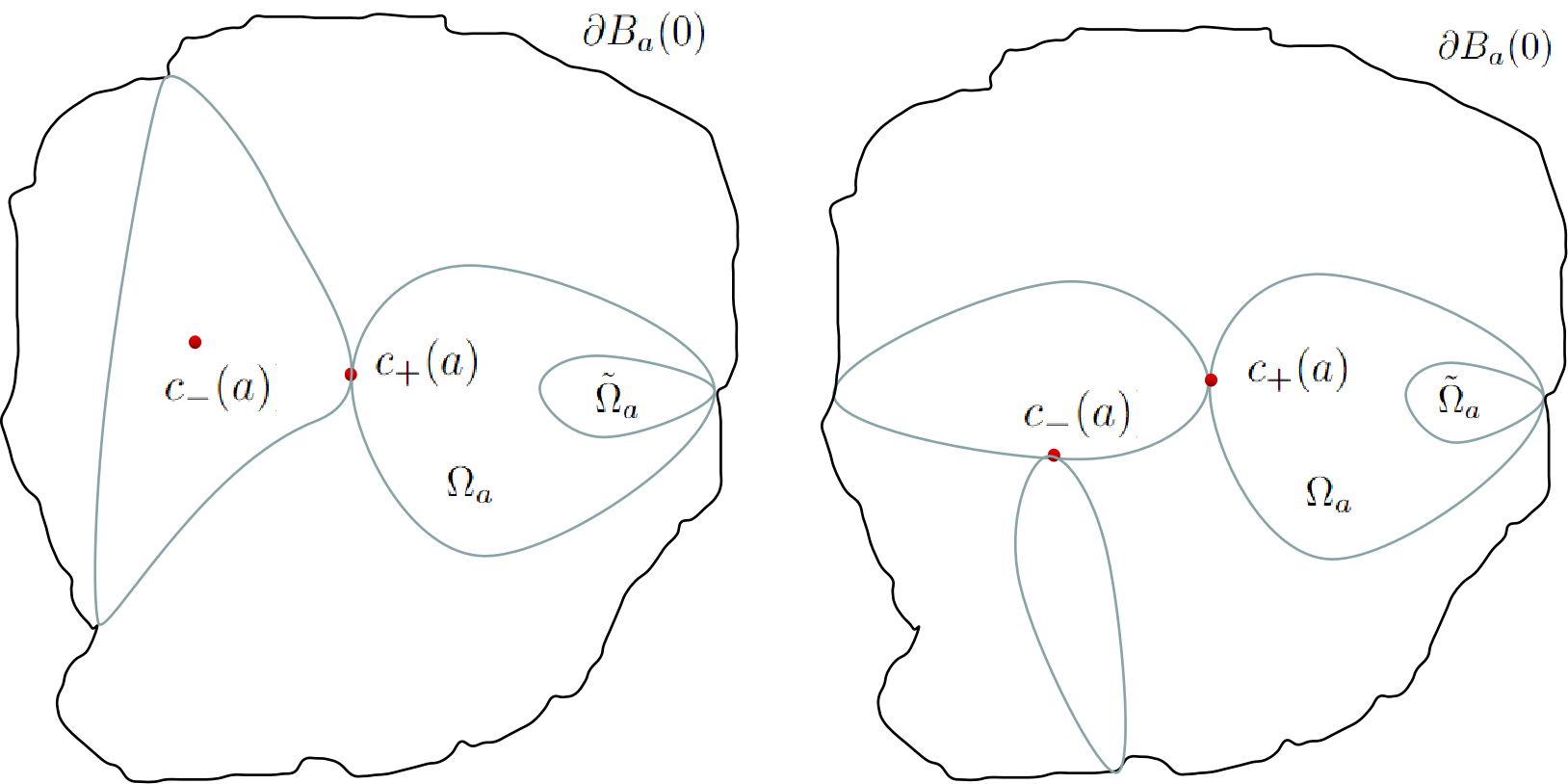} 
\caption{The preimage of $\Tilde{\Omega}_a$: $\Phi(a)\in W$ on the right and $\Phi(a)\in \partial W$ on the left.} 
\label{fig.preimage} 
\end{figure}

\begin{figure}[H] 
\centering 
\includegraphics[width=0.75\textwidth]{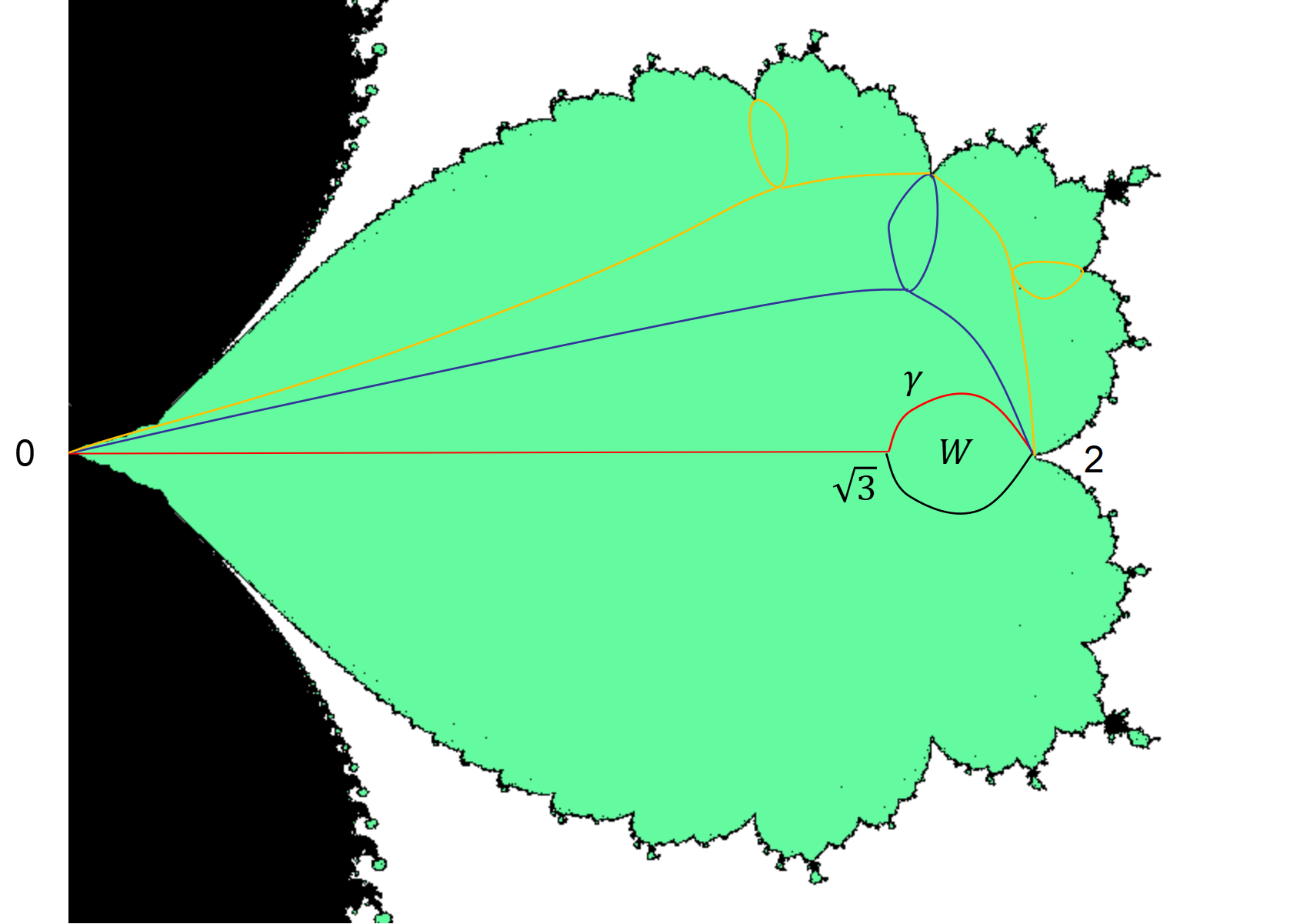} 
\caption{The connected component $\mathcal{U}_0$ of $\mathcal{H}_0$ with the curve $\gamma$ and the component $W$ depicted in Proposition \ref{PROP_para0}. The union of curves in red, blue and orange is equipotentials of level 0,1,2, denoted by $\mathcal{E}(0),\mathcal{E}(1),\mathcal{E}(2)$ respectively. See next section.} 
\label{fig.main.parabolic} 
\end{figure}

\section{Graphs and puzzles}\label{sec.puzzle}
\subsection{Dynamical rays}
This subsection is devoted to construct dynamical rays for $f_a,a\in\mathbb{C}^*$. For $t\in\mathbb{Q}/\mathbb{Z},r\textgreater 1$, the definitions of external rays $R^{\infty}_a(t)$ and equipotentials $E^{\infty}_a(r)$ for polynomials whose Julia set is connected are classical, see for example \cite{DH}:
\[R^{\infty}_a(t) := (\phi^{\infty}_a)^{-1}(\{re^{2\pi i t};\,r\textgreater 1\}),\,\,E^{\infty}_a(r) := (\phi^{\infty}_a)^{-1}(\{re^{2\pi i t};\,t\in[0,1)\}).\]

We will focus on the internal rays and internal equipotentials contained in the parabolic basin of $f_a$. Let us begin with a model of quadratic polynomial.
\paragraph{Model $P_{\frac{1}{4}}(z) = z^2+\frac{1}{4}$}\mbox\\

We keep using the notations in diagram (\ref{commu1}). Recall that \[\Omega = P_{\frac{1}{4}}(\Omega^0) = \phi^{-1}\{z;\,\,\mathfrak{Re}z\textgreater 1\}.\]
\begin{definition}
The equipotential (of depth $n\geq 0$) in $B$ is defined by $E(0) = \partial\Omega\setminus\{\frac{1}{4}\}$ and $E(n) = P^{-1}_{\frac{1}{4}}(E(n-1))$ for $n\geq 1$.
\end{definition}

Next following Roesch \cite{Roesch} we present here her construction of a "jigsawed" internal ray of angle $\theta = \frac{\pm1}{2^k-1}$ in the parabolic basin for the model $P_{\frac{1}{4}}$. Later in this subsection we generalize this construction for $f_a$. The motivation of introducing such internal rays was to create a infinite sequence of non-degenerated annuli around points on Julia set so that one may apply Yoccoz's Theorem which tells us that the puzzles shrink either to a single point, or a quasiconformal copy of some quadratic Julia set.

Let $\Xi^+$ (resp. $\Xi^-$) be the connected component of $B\setminus E(1)$ contained in the upper-half (resp. lower-half) plane. There are two inverse branches of $P$, $P_+:B\setminus\overline{\Omega}\longrightarrow \Xi^+$ and $P_-:B\setminus\overline{\Omega}\longrightarrow \Xi^-$ which extends continuously to the boundary. Let $\Delta^0 = \phi^{-1}\{z;\,\,0\leq\mathfrak{Re}z\leq1,\mathfrak{Im}z\geq0\}$, $\Delta^n = P_+^n(\Delta^0)$. One checks easily that
\[\Tilde{\phi}_+:\bigcup_{n=0}^{k-1}\Delta^n\longrightarrow\{z;\,\,-(k-1)\leq\mathfrak{Re}z\leq1,\mathfrak{Im}z\geq0\}\]
is a homeomorphism. Let $L$ be the straight line joining $-(k-1)$ and $1+i$. Fix a smooth curve (for example the semi-circle) $L':[0,1]\longrightarrow \mathbb{C}$ with $L'(0) = 0$, $L'(1) = i$ and $L'(t)\subset (0,1)\times(0,1)$ for $t\in (0,1)$. Morenover, notice that the connected component 
$\Tilde{\phi}_-:(\Omega^0)'\longrightarrow \mathbb{H}$ is also a homeomorphism extending to boundary, where $(\Omega^0)'$ is the bounded conneceted component of $\mathbb{C}\setminus E(1)$ with $-\frac{1}{2}$ on its boundary.
Define a segment $\delta$ in $B$ starting from $\Tilde{\phi}_+^{-1}(1+i)$ by
\[\delta = \Tilde{\phi}_+^{-1}(L)\bigcup P_+^{k-1}(\Tilde{\phi}_-(L'))\]
and set $R = \bigcup_{n=0}^{\infty}(P_+^{k-1}\circ P_-)^n(\delta)$

\begin{PandD}
Let $R$ be defined as above. Then the $k$ curves $R,...,P_{\frac{1}{4}}^{k-1}(R)$ are disjoint in $B\setminus \Omega$ and $R\subset P^k_{\frac{1}{4}}(R)$. Moreover $R$ lands at a $k-$periodic point of $P_{\frac{1}{4}}$ on $\partial B$. The curve $R\cap\Omega^c$ is called an internal ray with angle $\theta = \frac{1}{2^k-1}$ and denoted by $R(\theta)$. The internal rays $R(\theta')$ with angle $\theta'$ such that $\theta' = d^i\theta$ or $d^j\theta' = \theta$ are defined by $f^i(R(\theta))\cap\Omega^c$ or the connected component of $f^{-j}(R(\theta))$ whose landing point on $\partial B$ has cyclic order $\theta'$. Similarly one can construct internal ray with angle $\theta = \frac{-1}{2^k-1}$
\end{PandD}
\begin{proof}
By construction of $R$, $P^l_{\frac{1}{4}}(R),0\leq l\leq k-1$ are fixed by $P^k$ outside $B\setminus\Omega$. Notice that in $B\setminus\Omega^0$, $R,P_{\frac{1}{4}}(R),...,P_{\frac{1}{4}}^{k-2}(R)$ are contained in $\Xi_+$ while $P_{\frac{1}{4}}^{k-1}(R)$ is contained in $\Xi_-$ and $R,P_{\frac{1}{4}}(R),...,P_{\frac{1}{4}}^{k-1}(R)$ are disjoint in $\overline{\Omega^0}$. Hence these $k$ curves are disjoint in $B\setminus\Omega$. Finally, the Denjoy-Wolff Theorem applied to $P_+^{k-1}\circ P_-$ yields that $R$ converges to a $k-$periodic point.
\end{proof}

\begin{definition}
Let $a\not\in \mathcal{H}_0$, then $f_a:B^*_a(0)\longrightarrow B^*_a(0)$ is of degree two and therefore conformally conjugate to $P_{\frac{1}{4}}$ by the pull-back of Fatou coordinate. The internal rays of $P_{\frac{1}{4}}$ are transfered to internal rays in $B^*_a(0)$, denoted by $R^0_a(\theta)$.
\end{definition}

\begin{figure}[H] 
\centering 
\includegraphics[width=0.95\textwidth]{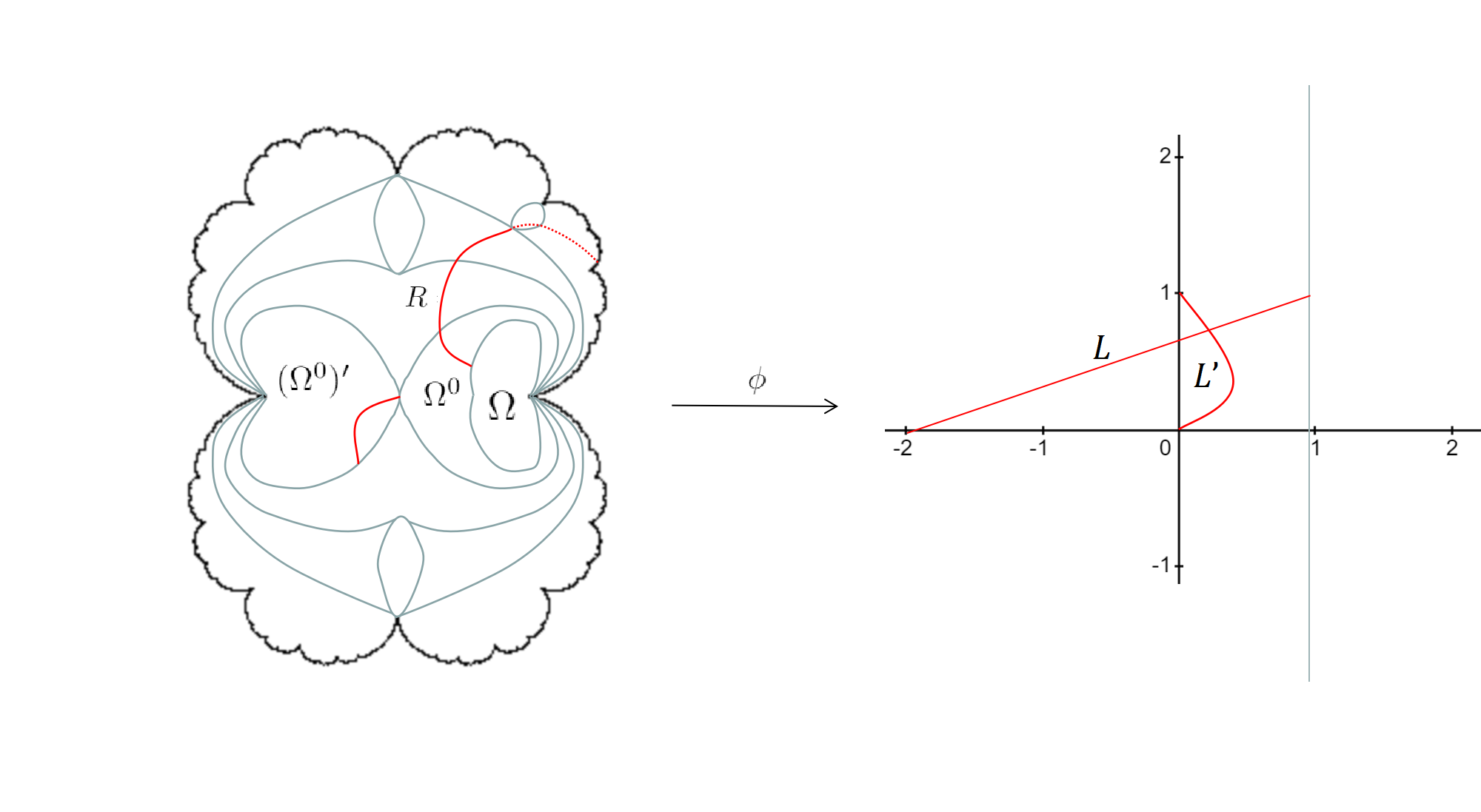} 
\caption{Illustration for the construction of $R$ when $k = 3$}\label{fig.internal} 
\label{Fig.main2} 
\end{figure}

Next we construct internal rays for $a\in \Tilde{\mathcal{U}}_0\cap S$, where $\Tilde{\mathcal{U}}_0$ as prescribed in Proposition \ref{PROP_para0}. In this situation the internal rays for $P_{\frac{1}{4}}$ can not completely be transformed to $B^*_a(0)$ since $B^*_a(0)$ contains the free critical point $c_-(a)$, but still one can construct $R^a_0(\theta)$ by a pull-back argument. 

Denote by $(\Omega^0_a)'$ the connected component of $f_a^{-1}(f_a(\Omega^0_a))$ which attaches $\Omega^0_a$ at $c_+(a)$. Hence $\overline{(\Omega^0_a)'\cup\Omega^0_a}$ forms a figure "eight" and cuts $B^*_a(0)$ into two connected components. Denote by $\Xi_-$ the one containing $c_-(a)$ and $\Xi_+$ the other. Hence $f_a:\Xi_+\longrightarrow B\setminus \overline{\Omega_a}$ is conformal and $f_a:\Xi_-\setminus \overline{f_a^{-1}(f_a(\Omega^0_a))}\longrightarrow B\setminus \overline{\Omega_a}$ is a 2-covering map ramified at $c_-(a)$. We denote them by $f^+_a,f^-_a$ for short. Both these two maps can be extended homeomorphically to the corresponding half-boundary of the "eight" which is contained in $\partial\Xi_{\pm}$ respectively. 

Now we construct $R^0_a(\theta)$ with $\theta = \frac{1}{2^k-1}$. The construction for angle $\frac{-1}{2^k-1}$ will be similar. In the right-half complex plane, let $I$ be the straight line joint $1+\frac{i}{k},0$ and let $L'$ be as in Figure \ref{fig.internal}. Notice that the restrictions of the Fatou coordinate $\phi^-_a := \phi_a|_{(\Omega^0_a)'},\phi^+_a :=\phi_a|_{\Omega^0_a}$ are conformal and set $l^{k-1}_a = \overline{\phi^{-1}_a|_{(\Omega^0_a)'}(L)}\cup\overline{\phi^{-1}_a|_{\Omega^0_a}(I)}$. Consider the pull back $(f^+_a)^{-1}(l^{k-1}_a)$ and link the two points $(f^+_a)^{-1}((\phi^-_{a})^{-1}(i)),(\phi^+_a)^{-1}(1+\frac{2i}{k})$ by $(\phi^+_a)^{-1}(I+\frac{i}{k})$. Denote by $l^{k-2}_{a} = (f^+_a)^{-1}(l^{k-1}_a)\cup(\phi^+_a)^{-1}(I+\frac{i}{k})$. Repeating this process until we get $l^1_a$ which is a piece-wise smooth curve joining $(\phi^+_a)^{-1}(1+i),(f^+_a)^{-k+1}((\phi^-_{a})^{-1}(i))$. Now suppose $v_-(a)\not\in l^1_a$. Then we can lift $l^1_a$ by $f^-_a$ which gives us a curve starting from $(\phi^-_{a})^{-1}(i)$ and hence $l^{k-1}_a$ is extended. Keep doing this pull-back argument under the hypothesis that $l^1_a$ will never meet $v_-(a)$ (we will give a sufficient condition for this later, see Lemma \ref{lem.holomotion.internal}) we will get a $k-$periodic cycle of curves $l^0_a,...,l^{k-1}_a$ converging to $\partial B^*_a(0)$. The curve $R^0_a(\theta) := l^0_a$ is called the internal ray with angle $\theta$.

\begin{remark}
One should notice that if the choice of the curve $L'$ in Figure \ref{fig.internal} does not depend on $a$, then the internal ray constructed above is uniquely defined and hence automatically admits a holomorphic motion. See Lemma \ref{lem.holomotion.internal} for a detailed description.
\end{remark}

\subsection{Parameter rays and landing properties}
\paragraph{Parameter rays in $\mathcal{H}_\infty$}\mbox\\

Define equipotentials and external rays in $\mathcal{H}_\infty$ by
\begin{equation}\label{eq.rays.Hinfty}
    \mathcal{E}_{\infty}(r) = \Phi_\infty^{-1}\{re^{2\pi it};\,\,t\in[0,1]\},\,\,\mathcal{R}_{\infty}(t) = \Phi_\infty^{-1}\{re^{2\pi it};\,\,r\textgreater 1\}.
\end{equation}

\begin{figure}[H] 
\centering 
\includegraphics[width=0.65\textwidth]{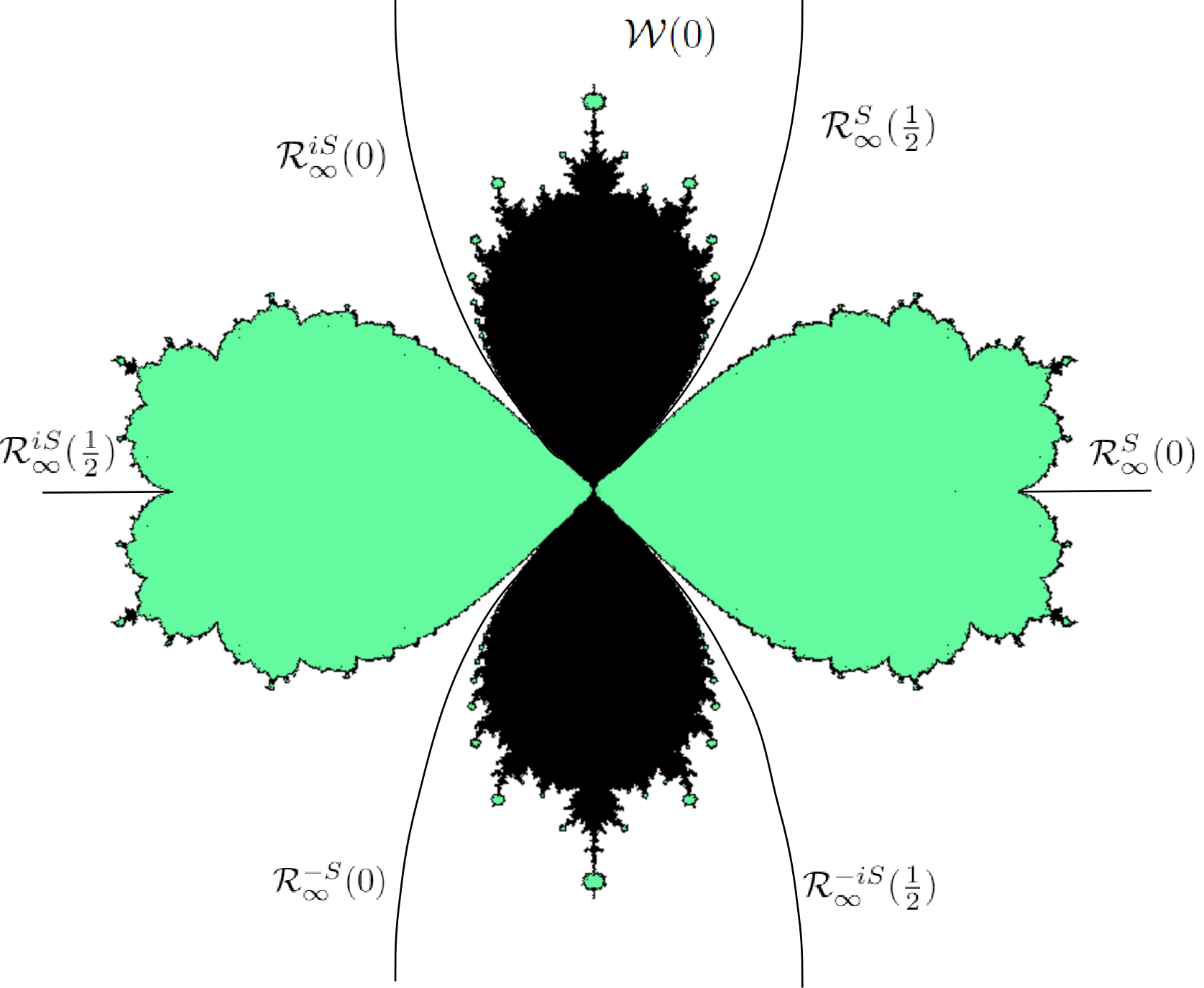} 
\caption{External rays $\mathcal{R}_{\infty}(0),\mathcal{R}_{\infty}(\frac{1}{2})$ among which the ones land at 0 bound two wakes $\mathcal{W}(0),-\mathcal{W}(0)$.} 
\label{fig.external} 
\end{figure}

\begin{remark}
Technically, for $t\in \mathbb{R}/\mathbb{Z}$, $\mathcal{R}_{\infty}(t)$ consists of three distinct simple curves since $\Phi_{\infty}$ is of degree 3. We call every such curve an external ray corresponding to the angle $t$. We write $\mathcal{R}_{\infty}^{S}(t),\mathcal{R}_{\infty}^{iS}(t),\mathcal{R}_{\infty}^{-S}(t),\mathcal{R}_{\infty}^{-iS}(t)$ respectively for pointing out the quadrant which this curve belongs to. See Figure \ref{fig.external} for an example.
\end{remark}

Next we prove some landing properties for the parameter external rays. We will need some classical stability results:
\begin{lemma}[\cite{DH}]\label{lem.stab.misiur}
 Let $P_{a}:\mathbb{C}\longrightarrow\mathbb{C}$ be an analytic family (parametrized by $a$) of polynomials of degree $d$, $t\in\mathbb{Q}/\mathbb{Z}$. Suppose that $R^{\infty}_{a_0}(t)$ lands at $x(a_0)\in J_{a_0}$ and $\overline{R^{\infty}_a(d^nt)},\forall n\geq 0$ contains no critical point of $P_{a_0}$ (in particular $R^{\infty}_a(d^nt)$ lands). If $x(a_0)$ is repelling (pre-)periodic, then there exists a neighborhood $U$ of $a_0$ such that $\forall a\in U$, $R^{\infty}_{a}(t)$ will land at a repelling (pre-)periodic point. Moreover there is a natural holomorphic motion preserving equipotential:
\[L:U\times\bigcup_{n\geq 0}\overline{R^{\infty}_{a_0}(d^nt)}\longrightarrow \mathbb{C},\, L(a,\overline{R^{\infty}_{a_0}(d^nt)}) = \overline{R^{\infty}_{a}(d^nt)}.\]
\end{lemma}

\begin{lemma}[stability of repelling petal]\label{lem.stab.parabo}
Let $P_a$, $R^{\infty}_{a_0}(d^nt)$ be as in Lemma \ref{lem.stab.misiur}. Suppose that $0$ is a common parabolic fixed point for $P_a$ with multiplier $e^{2\pi i\frac{p}{q}}$ and near 0 one has 
\[P^q_a(z) = z + \omega(a)z^{q+1} + o(z^{q+1}), \,\,\text{with }\omega(a_0)\not = 0.\]
If the landing point $x(a_0)$ of $R^{\infty}_{a_0}(t)$ is in the inverse orbit of 0, then there exists a neighborhood of $a_0$ such that $\forall a\in U$, $R^{\infty}_{a}(t)$ will land at some $x(a)$ which is also in the inverse orbit of 0 and there is a natural holomorphic motion preserving equipotential:
\[L:U\times\bigcup_{n\geq 0}\overline{R^{\infty}_{a_0}(d^nt)}\longrightarrow \mathbb{C},\, L(a,\overline{R^{\infty}_{a_0}(d^nt)}) = \overline{R^{\infty}_{a}(d^nt)}.\]
\end{lemma}
\begin{proof}
The prove is quite similar to that of the above lemma. First suppose that $R^{\infty}_{a_0}(t)$ is periodic. Let $l_{a_0}:[1,\infty]\longrightarrow\mathbb{C}$ with $l_{a_0}(r)$ representing the unique point on $R^{\infty}_{a_0}(t)$ with equipotential $r$. For $r$ large enough and $a$ close to $a_0$, the ray $R^{\infty}_a(t)$ is well-defined as $l_a(r) = (\phi^{\infty}_a)^{-1}\phi^{\infty}_{a_0}(l_{a_0}(r))$ ($\phi^{\infty}_a$ is the Böttcher coordinate at $\infty$) and thus in a neighborhood $U'$ of $a_0$, we have the holomorphic motion 
\[L:U'\times l_{a_0}[r,\infty)\longrightarrow\mathbb{C}, \,L(a,z) = (\phi^{\infty}_a)^{-1}\phi^{\infty}_{a_0}(z).\] 
We can pull back $L$ for $N$ steps with $f_{a_0},f_{a}$ by choosing the inverse branch corresponding to the cycle $\{d^nt\}$ as long as $a$ is in some small neighborhood $U\subset U'$ of $a_0$ ($U$ depends on $N$). This is possible since by assumption $R^{\infty}_{a_0}(d^nt)$ contains no critical point. Now fix $\forall r_0\textgreater 1$ such that $l_{a_0}(r_0)$ enters the repelling neighborhood of $x(a_0)$, and we shrink $U$ such that $L$ is defined for $l_{a_0}[r_0,\infty)$. By continuity dependence of repelling Fatou coordinate (since $\omega(a_0)\not=0$), if $a$ is sufficiently close to $a_0$, $l_a(r_0)$ will enter some repelling petal at $0$. Thus $R^{\infty}_a(t)$ lands at $0$ since it is attracted by $(f^q_a)^{-1}$ and hence all the other $R^{\infty}_a(d^nt)$ land at $J_a$ if we pull back $R^{\infty}_{a}(t)$.\\
If $R^{\infty}_a(t)$ is pre-periodic, i.e. $d^lt$ is periodic for some $l\geq 1$, then we apply the first case to $d^lt$ and pull back $R^{\infty}_a(d^lt)$.
\end{proof}

\begin{proposition}[rational parameter rays land]\label{propland}
Let $t\in \mathbb{Q}/\mathbb{Z}$. Then $\mathcal{R}_{\infty}(t)$ lands at some $a_0\in \partial \mathcal{C}_1$ and in the dynamical plane $R^{\infty}_{a_0}(t)$ lands at some parabolic or repelling $k-$(pre-)periodic point $x({a_0})$. Moreover if $a_0\not = 0$ and $x({a_0})$ is repelling (pre-)periodic or in the inverse orbit of $0$ (under $f_{a_0}$), then $x({a_0}) = v_-(a_0)$.
\end{proposition}
\begin{proof}
Let $a_0 \in\partial \mathcal{C}_1$ be an accumulation point of $\mathcal{R}_{\infty}(t)$. Then in particular $J_{a_0}$ is connected, then $R^{\infty}_{a_0}(t)$ lands at some parabolic or repelling $k-$(pre-)periodic point $x({a_0})$ (see for example \cite{DH}). Now if the only accumulation point of $\mathcal{R}_{\infty}(t)$ is 0, then we are done, so we assume that $a_0\not = 0$. We consider 3 cases respectively.
\begin{itemize}
    \item $x(a_0)$ is in the inverse orbit of 0, i.e. $f_{a_0}^n(x(a_0)) = 0$. Then $R^{\infty}_{a_0}(3^nt)$ lands at 0. Let $x(a_n)$ be the landing point of $R^{\infty}_{a_n}(t)$. By Lemma \ref{lem.stab.parabo}, we have
    \[x(a_0) =\lim_{n\to\infty}x(a_n) = \lim_{n\to\infty}f_{a_n}(c_-(a_n)) = v_-(a_0).\]
    Thus $a_0$ verifies $f_{a_0}^{n+1}(c_-(a_0)) = 0$ and there are only finitely many solutions.
    
    \item $x(a_0)$ is (pre-)parabolic but not in the inverse orbit of 0. Then there exists $l\geq 0,k\geq 1$ (only depend on $t$) such that $f^l(R^\infty_{a_0}(t))$ is fixed by $f_{a_0}^k$. By the snail lemma, $(a_0,x(a_0))$ verifies \begin{equation}\label{eq.parabolic1}
        f^{k+l}_{a}(z) = f_a^l(z),\,\,(f^k_a)'(f_a^l(z)) = 1.
    \end{equation} 
    (\ref{eq.parabolic1}) defines a non-trivial algebraic variety (not equal to $\mathbb{C}^2$) and hence consists of only finitely many irreducible components of dimension 1 and 0 (i.e. points). We claim that the only irreducible component of dimension 1 is $z = 0$. Indeed, suppose there is another such component $X$. Then $X$ is unbounded. Consider the projections $\pi_1: X\longrightarrow \mathbb{C}$, $\pi(a,z) = a$ and $\pi_2: X\longrightarrow \mathbb{C}$, $\pi(a,z) = z$. Then at least one of $\pi_i(X),i=1,2$ is unbounded. If $\pi_1(X)$ is unbounded, let $(a_n,z_n)\subset X$ with $a_n\to \infty$, then for $n$ large enough, $z_n = 0$, so $X$ has to be $z = 0$. If $\pi_2(X)$ is unbounded, let $(a_n,z_n)\subset X$ with $z_n\to \infty$. If $(a_n)$ also tends to $\infty$ then we are in the precedent case; if not, then the Böttcher coordinate of $f_{a_n}$ at $\infty$ is stable, which implies that for $n$ large enough, $z_n\in B_{a_n}(\infty)$, contradicting the assumption that $(a_n,z_n)$ verifies (\ref{eq.parabolic1}).
    
    \item $x(a_0)$ is (pre-)repelling. Let us say $f_{a_0}^{k+l}(x(a_0)) = f^l_{a_0}(x(a_0))$. This case is similar to the first case. The only difference is that here we apply Lemma \ref{lem.stab.misiur}. By the same argument of the first case we will still get a polynomial equation of $a_0$: $f_{a_0}^{k+l+1}(c_-(a_0)) = f^{l+1}_{a_0}(c_-(a_0))$. Therefore there are only finite many possible solutions.
\end{itemize}
To conclude, the above analysis shows that the accumulation set of $\mathcal{R}_0(t)$ is finite, then it reduces to a single point since the accumulation set of $\mathcal{R}_0(t)$ is connected.
\end{proof}

\begin{corollary}
$\mathcal{R}^{S}_{\infty}(\frac{1}{2})$ and $\mathcal{R}^{iS}_{\infty}(0)$ both land at 0.
\end{corollary}
\begin{proof}
(See also in \cite{Nakane}) We prove for $\mathcal{R}^{S}_{\infty}(\frac{1}{2})$ and the other is similar. Suppose $\mathcal{R}^{S}_{\infty}(\frac{1}{2})$ lands at $a_0\not=0$. Then $R^{\infty}_{a_0}(\frac{1}{2})$ lands at a parabolic or repelling fixed point. The two fixed points of $f_{a_0}$ are $0,-a_0$ with multiplier $1,1+a_0^2$ respectively. If $R^{\infty}_{a_0}(\frac{1}{2})$ lands at $0$, then by Proposition \ref{propland}, $v_-(a_0) = 0$ and since we are considering $\mathcal{R}^{S}_{\infty}(\frac{1}{2})$, $a_0 = 2$. However in the dynamical plan of $f_{2}$, only $R^{\infty}_2(0)$ lands at 0. So $R^{\infty}_{a_0}(\frac{1}{2})$ lands at $-a_0$ which must be repelling. Then again by Proposition \ref{propland}, $-a_0 = v_-(a_0)$. Since $R^{\infty}_{a_0}(\frac{1}{2})$ is fixed by $f_{a_0}$, either $v_-(a_0) = \Tilde{c}_-(a_0)$ or $c_-(a_0)$, where $\Tilde{c}_-(a_0)$ denotes the co-critical point. But neither are possible for otherwise $c_-(a_0)$ will be periodic while it is in the Julia set.
\end{proof}

From this we define the wakes at 0:
\begin{definition}\label{def.wake0}
 The \textbf{positive wake} at 0, denoted by $\mathcal{W}(0)$, is defined to be the open region bounded by $\mathcal{R}^{S}_{\infty}(\frac{1}{2})$ and $\mathcal{R}^{iS}_{\infty}(0)$ which contains the positive imaginary axis. The \textbf{negative wake} at 0 is defined to be $-\mathcal{W}(0)$.
\end{definition}

\begin{lemma}\label{lemwake}
Let $a\in\mathbb{C}^*$. Then in the dynamical plane $R^{\infty}_a(0),R^{\infty}_a(\frac{1}{2})$ both land at 0 if and only if $a\in\mathcal{W}(0)\cup(-\mathcal{W}(0))$.
\end{lemma}
\begin{proof}
Let $W\subset\mathbb{C}^*$ be the set of parameters such that $R^{\infty}_a(0)$ lands at 0. In the proof of the first case of Propostion \ref{propland} we see that $W$ is open. Moreover $\partial W\subset\overline{\mathcal{R}^S_{\infty}(0)}\cup\overline{\mathcal{R}^{iS}_{\infty}(0)}\cup\overline{\mathcal{R}^{-S}_{\infty}(0)}$. Indeed, if $R^{\infty}_a(0)$ does not land at 0, then either it crashes on the critical point $c_-(a)$, or it lands at the other fixed point $-a$. If we are in the first case, then $v_-(a)\in R^{\infty}_{a}(0)$ and hence $a\in\mathcal{R}_{\infty}(0)$. If we are in the latter case, then $-a$ is repelling since the only parabolic fixed point with multiplier 1 is 0. By Lemma\ref{lem.stab.misiur}, $a\not\in\partial W$. Similarly one can prove that the set of parameters such that $R^{\infty}_{a}(\frac{1}{2})$ lands at 0, denoted by $W'$, is open and with boundary contained in $\overline{\mathcal{R}^S_{\infty}(\frac{1}{2})}\cup\overline{\mathcal{R}^{iS}_{\infty}(\frac{1}{2})}\cup\overline{\mathcal{R}^{-iS}_{\infty}(\frac{1}{2})}$. One may verify easily by symmetricity that $R^{\infty}_a(0)$ lands at 0 while $R^{\infty}_a(\frac{1}{2})$ not if $a\in\mathbb{R}^+$ and vice versa if $a\in\mathbb{R}^-$. Hence the lemma follows.
\end{proof}

\begin{lemma}\label{lemparaland}
Let $a\in \overline{\mathcal{W}(0)\cup S}$ be such that $f^{n+1}_a(c_-(a))=0$, $f^{n}_a(c_-(a))\not=0$. If $a\not\in\mathcal{W}(0)$ (resp. $a\in\mathcal{W}(0)$), then there exists a (resp. two) unique $t\in \mathbb{Q}/\mathbb{Z}$ (resp. $t,t'$) such that $R^\infty_a(t)$ (resp. $R^\infty_a(t),R^\infty_a(t')$) lands at $v_-(a)$ in the dynamical plan and $\mathcal{R}_{\infty}(t)$ (resp. $\mathcal{R}_{\infty}(t),\mathcal{R}_{\infty}(t')$) lands at $a$ in the parameter plan.
\end{lemma}

\begin{proof}
We prove for $a\not\in{\mathcal{W}(0)}$, the left case is similar. Since $c_-(a)$ will eventually hit 0, the Julia set $J_a$ is connected, hence there exists some $t\in \mathbb{Q}/\mathbb{Z}$ with $3^nt \equiv 1 \pmod{\mathbb{Z}}$ such that $R^{\infty}_{a}(t)$ lands at $v_-(a)$. It is clear that $R^{\infty}_{a_0}(t)$ is the unique ray in the dynamical plan landing at $v_-(a)$: notice that $f_a$ is univalent near $f^k_a(c_-(a))$ for $k\geq1$. \\
Next we prove the landing property in the parameter space. Suppose $a_0$ verifies the assumptions in the lemma. Let $l_{a_0}:[1,\infty]\longrightarrow\mathbb{C}$ with $l_{a_0}(r)$ representing the unique point on $R^{\infty}_a(t)$ with equipotential $r$. Recall the holomorphic motion given by Lemma \ref{lem.stab.parabo}
\[L:W\times R^{\infty}_{a_0}(t)\longrightarrow \mathbb{C},\,\,(a,l_{a_0}(r))\mapsto l_a(r).\]
Consider the holomorphic function $F_r(a) := v_-(a) - l_a(r)$. Then $F_1(a_0) = 0$. By Rouché's theorem, for every neighborhood $U\subset W$ of $a_0$, there exists $r$ close enough to 1 and $a\in U$ such that $F_r(a) = 0$. This means that $\mathcal{R}_{\infty}(t)$ accumulates to $a$ and by Proposition \ref{propland}, it actually lands at $a$. Now suppose that there is another ray $\mathcal{R}_{\infty}(t')$ with $t$ rational landing at $a$. Then by Proposition \ref{propland}, $R^{\infty}_{a}(t')$ lands at $v_-(a)$, contradicts to the uniqueness of $R^{\infty}_{a}(t)$.
\end{proof}

Adapting the same strategy as above (applying Lemma \ref{lem.stab.misiur}), we can prove the following landing property for parameters of Misiurwicz type:
\begin{lemma}\label{lem.para.land.Misiur}
Let $a\in \overline{\mathcal{W}(0)\cup S}$ be such that $v_-(a)$ is repelling (pre-)periodic. Let $R^{\infty}_a(\eta)$ land at this periodic point. Then there exists a unique $t\in\mathbb{Q}/\mathbb{Z}$ with $3^nt = \eta$ for some $n\geq 0$ such that $R^{\infty}_a(t)$ lands at $v_-(a)$ and $\mathcal{R}_{\infty}(t)$ lands at $a$.
\end{lemma}

\paragraph{Parameter rays in $\mathcal{H}_0$}\mbox\\ 

Define equipotentials and internal rays contained in $\mathcal{H}_0\cap{\overline{S}}$ by
\begin{equation}\label{eq.rays.H0}
    \mathcal{E}(0) = [0,\sqrt{3}]\cup \gamma,\,\,\mathcal{E}(n) = \overline{\Phi^{-1}(E(n)\cap B)}\,\,(n\geq 1);\,\,\mathcal{R}(\theta) = \Phi^{-1}(R(\theta)).
\end{equation}

\begin{proposition}\label{propequiH0}
For $n\geq 1$, $\mathcal{E}(n-1)\cap\partial\mathcal{H}_0$ consists of finitely many points. Let $\Tilde{\mathcal{E}}(n)$ denote the quotient space of $\mathcal{E}(n)$ by gluing 0 and 2, then $\Tilde{\mathcal{E}}(n)$ is homeomorphic to $E(n)$. In particular $(\mathcal{E}(n-1)\cap \partial\mathcal{H}_0)\subset(\mathcal{E}(n)\cap \partial\mathcal{H}_0)$. Let $a\in (\mathcal{E}(n)\cap \partial\mathcal{H}_0)\setminus(\mathcal{E}(n-1))\cap \partial\mathcal{H}_0)$, then $f^{n+1}_a(c_-(a))=0$, $f^{n}_a(c_-(a))\not=0$.
\end{proposition}

\begin{proof}
The proof proceeds in three parts:
\begin{itemize}
    \item $\#(\mathcal{E}(n-1)\cap\partial\mathcal{H}_0)\textless\infty$. Let $a_0\in \mathcal{E}(n)\cap{\partial\mathcal{H}_0},\,\,n\geq0$  and $a\not =0$, $(a_m)\subset \mathcal{E}(n)\cap\mathcal{H}_0$ a sequence tending to $a_0$. By definition of $\mathcal{E}(n)$, $f_{a_m}^{n+1}(c_-(a_m))\in \overline{f_{a_m}(\Omega^0_{a_m})}$. By continuity of Fatou coordinate (since $a_0\not=0$), $\phi_{a_m}(f_{a_m}^{n+1}(c_-(a_m)))\to \infty$, therefore $f_{a_0}^{n+1}(c_-(a_0)) = 0$. It has only finitely many solutions.
    
    \item $\Tilde{\mathcal{E}}(1)$ is homeomorphic to $E(1)$. By the first step the only possible accumulation points of ${\Phi^{-1}(\partial\Omega^0)}$ in $\partial\mathcal{H}_0$ are $\{0,2\}$. Clearly both $\{0,2\}$ are reached by $\Phi^{-1}(\partial\Omega^0)$ since $\Phi^{-1}(\partial\Omega^0),\gamma,[0,\sqrt{3}]$ surround a region which is mapped under $\Phi$ to $\Omega^0\setminus\Omega$. Let $\Tilde{\Omega}^1$ be the region bounded by $\mathcal{E}(1)\setminus\Phi^{-1}(\partial\Omega^0)$. Again by the first step, $\partial\Tilde{\Omega}^1$ lands either at 0, or at some $a_0$ such that $f_{a_1}^2(c_-(a_1)) = 0$ but $f_{a_1}(c_-(a_1)) \not= 0$. Suppose that the landing point is 0, then $\Phi^{-1}(E(1)\cap\mathbb{L})$ bounds a simply connected region $W_1$ in the parameter plane ($\mathbb{L}$ is the lower-half plane), so $W_1\subset \mathcal{C}_1$ must contain parts of $\partial H_0$ and this contradicts Lemma \ref{lemboundaryH}. 
    \item $\Tilde{\mathcal{E}}(n)$ is homeomorphic to $E(n)$. By the preceeding step and Lemma \ref{lemparaland}, $\mathcal{R}_{\infty}(\frac{1}{3})$ lands at $a_1$. $S\setminus\ (\mathcal{R}_{\infty}(\frac{1}{3})\cup\mathcal{E}(1))$ has two unbounded connected components $V_1,V'_1$ with $V_1$ containing 0 on its boundary while the other not. Hence $\Phi^{-1}(E(2)\cap-\mathbb{L})$ will not accumulate at 0. Denote by $\Tilde{\Omega}^2\subset V_1$ the region bounded by $\mathcal{E}(2)\setminus\Phi^{-1}(\partial(\Omega^1))$. A similar argument as step 2 yields that $\partial\Tilde{\Omega}^2$ lands at some $a_2$ with $f_{a_2}^3(c_-(a_2)) = 0$, $f^{2}_{a_2}(c_-(a_2)) \not= 0$. This implies that $\Tilde{\mathcal{E}}(2)$ is homeomorphic to $E(2)$. Again by Lemma \ref{lemparaland}, $\mathcal{R}_{\infty}(\frac{4}{9})$ lands at $a_2$, hence $S\setminus\ (\mathcal{R}_{\infty}(\frac{4}{9})\cup\mathcal{E}(2))$ has two unbounded connected components $V_2,V'_2$ with $V_2$ containing 0 on its boundary while the other not. Inductively, we can repeat the above procedures to construct the corresponding $V_n$ with $V_n$ defined by the unbounded connected component of $S\setminus(\mathcal{R}_{\infty}(\frac{3^{n}-1}{2\cdot3^n})\cup\mathcal{E}(n))$ containing 0 on its boundary (the angle $t_n = \frac{3^{n}-1}{2\cdot3^n}$ is obtained by induction: $\frac{1}{2}\textgreater t_n\textgreater t_{n-1}$) and show that $\left(\mathcal{E}(n+1)\setminus\Phi^{-1}(\partial(\Omega^n))\right)\cap V_n$ does not accumulate at 0. The Proposition follows for all $n$.
\end{itemize}

\end{proof}

\begin{proposition}\label{prop.para.land.Misur}
Let $\theta$ be such that $2^n\theta = \frac{\pm1}{2^k-1}$, $k\textgreater 1,n\geq 0$. Then the parameter ray $\mathcal{R}_0(\theta)$ lands at some $a_0\in \partial \mathcal{H}_0\setminus\{0\}$ and the corresponding dynamical ray $R^0_{a_0}(\theta)$ lands at some (pre-)parabolic or (pre-)repelling $k-$periodic point $x(a_0)\in \partial B^*_{a_0}(0)$. Moreover if $x(a_0)$ is (pre-)repelling, then $R^0_{a_0}(\theta)$ lands at $v_-(a_0)$.
\end{proposition}
\begin{proof}
Let $a_0 \in \partial \mathcal{H}_0$ be an accumulation point of $\mathcal{R}_0(\theta)$. From Lemma \ref{lemparaland} and Proposition \ref{propequiH0} we see that $\overline{\mathcal{R}_0(\theta)}$ is separated from $0$, therefore $a_0\not= 0$. Let $x(a_0)$ be the $k-$periodic landing point of $R^0_{a_0}(\theta)$. Since the dynamical ray $R^0_{a_0}(\theta)$ is fixed by $f^k$,  by the snail lemma, $x(a_0)$ is a fixed point of $f^k$ either parabolic with multiplier 1 or repelling. Similarly as Proposition \ref{propland}, it suffices to prove that there are only finitely many possible accumulation value at $\partial\mathcal{H}_0$ of $\mathcal{R}_0(\theta)$. In fact the idea of the proof goes the same as Proposition \ref{propland}: $x(a_0)$ parabolic corresponds to the second case in the proof of Proposition \ref{propland} and $x(a_0)$  repelling corresponds to the third case.
\end{proof}

\paragraph{Parameter rays in $\mathcal{H}_n$ ($n\geq1$)}\mbox\\ 

Let $\mathcal{U}$ be a connected component of $\mathcal{H}_n$. Define equipotentials and internal rays contained in $\overline{\mathcal{U}}$ by
\begin{equation}\label{eq.rays.Hn}
    \mathcal{E}_{\mathcal{U}}(n) = \overline{\Phi_{\mathcal{U}}^{-1}(E(n)\cap B)}\,\,(n\geq 0);\,\,\mathcal{R}_{\mathcal{U}}(\theta) = \Phi_{\mathcal{U}}^{-1}(R(\theta)).
\end{equation}

\begin{figure}[H] 
\centering 
\includegraphics[width=0.7\textwidth]{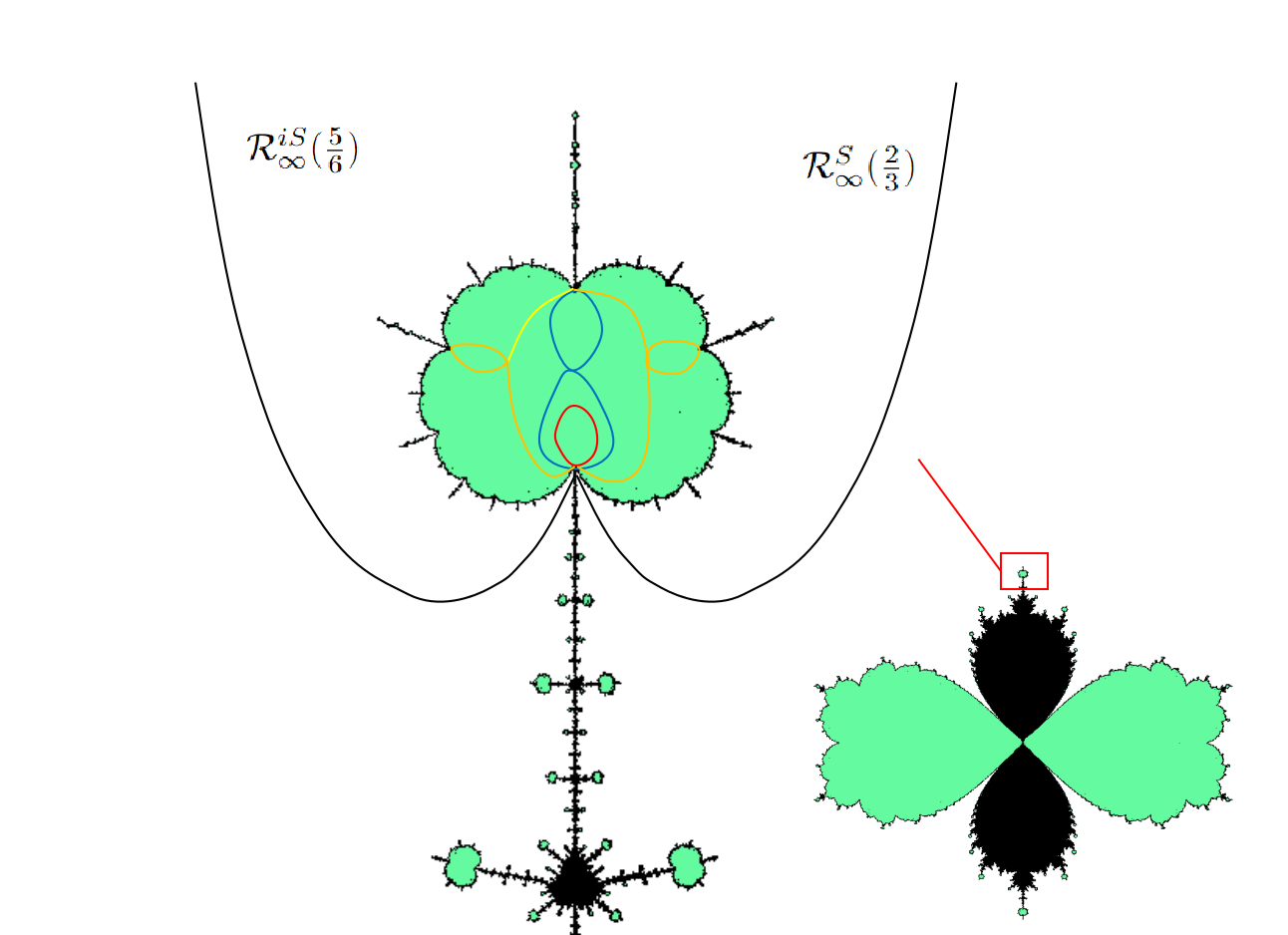} 
\caption{A zoom of a connected component $\mathcal{U}_1$ of $\mathcal{H}_1$. $\mathcal{U}_1$ is a copy of the cauliflower due to Proposition \ref{prop.para.Hn}. Every Misiurewicz parabolic parameter admits two rays landing since $\mathcal{U}_1\subset\mathcal{W}(0)$. For example the two external rays $\mathcal{R}^S_{\infty}(\frac{2}{3}),\mathcal{R}^{iS}_{\infty}(\frac{5}{6})$ land at $a_0$ such that $f_{a_0}(v_-(a_0)) = 0$. The union of curves in red, blue, orange inside $\mathcal{U}_1$ is the equipotential $\mathcal{E}_{\mathcal{U}_1}(0),\mathcal{E}_{\mathcal{U}_1}(1),\mathcal{E}_{\mathcal{U}_1}(2)$ respectively.} 
\label{Fig.main2} 
\end{figure}

\begin{proposition}\label{propequiHn}
$0\not\in\overline{\mathcal{U}}$. ${\mathcal{E}}(n)$ is homeomorphic to $E(n)$. In particular $(\mathcal{E}(n-1)\cap \partial\mathcal{H}_0)\subset(\mathcal{E}(n)\cap \partial\mathcal{H}_0)$. Let $a\in (\mathcal{E}(n)\cap \partial\mathcal{H}_0)\setminus(\mathcal{E}(n-1))\cap \partial\mathcal{H}_0)$, then $f^{n+1}_a(c_-(a))=0$, $f^{n}_a(c_-(a))\not=0$.
\end{proposition}
\begin{proof}
By symmetricity of $\mathcal{H}_n$ we may suppose that $\mathcal{U}\subset \mathcal{W}(0)\cup S$. Essentially once we prove that 0 is not on the boundary of $\mathcal{U}$, the homeomorphism between ${\mathcal{E}}(n)$ and $E(n)$ is immediate. Suppose not, then there are two cases to consider respectively:
\begin{itemize}
    \item $\mathcal{U}\subset \mathcal{W}(0)$. Notice that $-a$ is a fixed point of $f_a$. When $a\in\mathcal{U}$, $J_a$ is connected, hence some rational external ray $R^{\infty}_a(t)$ lands at $-a$. Clearly $-a$ must be a repelling. By stability of landing property, $t$ does not depend on $a$ for $a\in\mathcal{U}$. Suppose $t^k = t \mod \mathbb{Z}$ ($k\textgreater 1$ since $R^{\infty}_a(\frac{1}{2})$ lands at 0). Therefore $R^{\infty}_0(t)$ lands at some repelling periodic point with period $\textgreater 1$ of $f_0$. However by stability of landing property, $R^{\infty}_a(t)$ also lands at a repelling periodic point with period $\textgreater 1$, a contradiction.
    \item $\mathcal{U}\subset S\setminus\overline{\mathcal{W}(0)}$. If the landing point (denoted by $a_0$) of $\mathcal{E}_{\mathcal{U}}(0)$ is not 0, then by continuity of Fatou coordinate $f^{n+1}_a(c_-(a)) = 0$. If $a_0 = 0$, then $\mathcal{E}_{\mathcal{U}}(1)$ must have a landing point $a_1$ other than 0. Indeed, suppose the contrary, then $\Phi^{-1}_{\mathcal{U}}(E(1))\cap\mathbb{L}$ or $\Phi^{-1}_{\mathcal{U}}(E(1))\cap-\mathbb{L}$ ($\mathbb{L}$ is the lower-half plane) will bound a simply connected region $W\subset \mathcal{C}_1$ containing points in $\partial\mathcal{U}$. This contradicts Lemma \ref{lemboundaryH}. Hence by Lemma \ref{lemparaland} there exists an external parameter ray $\mathcal{R}_{\infty}(t)$ with $t\textless\frac{1}{2}$ landing on $\partial\mathcal{U}$. However, recall that in Proposition \ref{propequiH0} step 3 we have constructed a sequence of connected open sets $V_n$ with 0 and $\mathcal{R}_{\infty}(\frac{3^n-1}{2\cdot3^n})$ contained in its boundary. Therefore $\overline{\mathcal{U}}\cup\mathcal{R}_{\infty}(t)\subset \bigcap V_n$. But $t\textless\frac{1}{2}$ while $\lim_n\frac{3^n-1}{2\cdot3^n} = \frac{1}{2}$, which leads to a contradiction.
\end{itemize}
\end{proof}
\begin{remark}\label{rem.para.land.Misur}
Since $0\not\in\partial\mathcal{U}$, We have the similar landing properties for $\mathcal{R}_{\mathcal{U}}(\theta)$ as those of $\mathcal{R}_{0}(\theta)$ in Proposition \ref{prop.para.land.Misur}.
\end{remark}

\subsection{Parameter graphs and puzzles}\label{sub.sec.paragraph}
For the reason of symmetricity, we will only construct the parameter graphs and puzzle pieces in $\overline{\mathcal{W}(0)\cup S}$ (which means that all the external rays, equipotentials appearing below are contained in $\overline{\mathcal{W}(0)\cup S}$). Recall the definitions of parameter rays, equipotentials in (\ref{eq.rays.Hinfty}) (\ref{eq.rays.H0}) and (\ref{eq.rays.Hn}).

\paragraph{Parameter graphs $\mathcal{Y}_n$ and puzzles $\mathcal{Q}_n$}\mbox{}\\
Fix some $r\textgreater 1$, for each $n\geq0$ define the graph adapted for parameters of \textbf{Misiurewicz parabolic type}.
\[\mathcal{Y}_n = \bigcup_{k=0}^{n}\left(\bigcup_{\mathcal{U}_k\subset\mathcal{H}_k}\mathcal{E}_{\mathcal{U}_k}(n-k)\right)\cup\left(\bigcup_{t\in T_n}\overline{\mathcal{R}_{\infty}(t)}\right)\cup\mathcal{E}_{\infty}(r^{1/3^n})\]
where 
\[T_n = T^{0}_n\cup (T^{\frac{1}{2}}_n\cap \{\frac{1}{2}\leq t\leq1\}),\,\,T^{0}_n = \{t;\,3^nt = 1\};\,T^{\frac{1}{2}}_n = \{t;\,3^nt = \frac{1}{2}\}.\]
By Lemma \ref{lemparaland}, Proposition \ref{propequiH0} and Proposition \ref{propequiHn}, $\mathcal{Y}_n$ is connected, hence every connected component of $({\mathcal{W}(0)\cup S})\setminus\mathcal{Y}_n$ is simply connected. We call such a connected component $\mathcal{Q}$ a puzzle piece associated to $\mathcal{Y}_n$ (or of depth $n$) if it is bounded and $\partial \mathcal{Q}\cap\mathcal{E}_{\infty}(r^{1/3^n})\not =\emptyset$. We denote it by $\mathcal{Q}_n$ in the sequel. By construction, every $\mathcal{Q}_{n+1}$ is contained in a unique puzzle piece $\mathcal{Q}_{n}$ of depth $n$.

\begin{lemma}\label{lem.holomotion.external.parabo}
For every $n\geq 0$, Let $\mathcal{Q}$ be a connected component of \[\left(\mathcal{W}(0)\cup S\setminus\bigcup_{t \in T_{n}} \overline{\mathcal{R}_{\infty}(t)}\right)\] 
Then there is a dynamical holomorphic motion
$L_n: \mathcal{Q}\times\left(\bigcup_{t \in T_{n+1}}\overline{R^\infty_{a_0}(t)}\right)\longrightarrow \mathbb{C}$,
where any $a_0\in \mathcal{Q}$ can be chosen to be the base point. By saying dynamical one means that (denote $L_n(a,z)$ by $L^a_n(z)$)
\[\xymatrix{
 \underset{t \in T_{n+2}}{\bigcup}\overline{R^\infty_{a_0}(t)}\ar[r]^{L^a_{n+1}}\ar[d]_{f_{a_0}} &    \underset{t \in T_{n+2}}{\bigcup}\overline{R^\infty_{a}(t)} \ar[d]^{f_a}\\
 \underset{t \in T_{n+1}}{\bigcup}\overline{R^\infty_{a_0}(t)} \ar[r]^{L^a_{n}} &   \underset{t \in T_{n+1}}{\bigcup}\overline{R^\infty_{a}(t)}}\]
\end{lemma}
\begin{proof}
By assumption, $a_0\not\in\bigcup_{t \in T_{n}} \overline{\mathcal{R}_{\infty}(t)}$ ensures that $\bigcup_{t \in T_{n+1}}\overline{R^\infty_{a_0}(t)}$ contains no critical point, therefore ${R}^{\infty}_{a_0}{\infty}(t)$ will land at a parabolic or repelling pre-fixed point $x(a_0)$. Moreover if $x(a_0)$ is parabolic, then its multiplier is 1 since ${R}^{\infty}_{a_0}{\infty}(t)$ will be eventually fixed by $f_{a_0}$. While the two finite fixed points of $f_{a_0}$ are $0$ and $-a_0$ with multiplier $1$ and $1+a_0^2$ respectively. Hence by Lemma \ref{lem.stab.misiur} and \ref{lem.stab.parabo}, $L_n$ can be defined near $a_0$. Notice that
$(\mathcal{W}(0)\cup S\setminus\bigcup_{t \in T_{n}} \overline{\mathcal{R}_{\infty}(t)})$ is simply connected since two different rays $\mathcal{R}_{\infty}(t),\mathcal{R}_{\infty}(t')$ land at different points by Lemma \ref{lemparaland}, therefore $L_n$ can be extended to $\mathcal{Q}$.
\end{proof}

There are also dynamical holomorphic motions of equipotentials and the proof is similar (simpler) as above:
\begin{lemma}\label{lem.holomotion.equi} 
For $n\geq0$, let $\mathcal{Q}$ be the unbounded connected component of \[(\mathcal{W}(0)\cup S)\setminus\bigcup_{k=0}^{n}\left(\bigcup_{\mathcal{U}_k\subset\mathcal{H}_k}\mathcal{E}_{\mathcal{U}_k}(n-k)\right)\]
and $\mathcal{Q}'$ be the bounded connected component of $(\mathcal{W}(0)\cup S)\setminus\mathcal{E}_{\infty}(r^{1/3^n})$. Then there are dynamical holomorphic motions\footnote{Here we use the abuse of notations $L_n$ representing different holomorphic motions.}:
\[L_n: \mathcal{Q}\times \bigcup_{j=0}^{n+1}f^{-j}_{a_0}(f_{a_0}(\Omega_{a_0}))\longrightarrow \mathbb{C},\,\, L_n: \mathcal{Q}'\times \bigcup_{j=0}^{n+1} E^{\infty}_{a_0}(r^{1/3^j})\longrightarrow \mathbb{C}.\]
\end{lemma}

\paragraph{Parameter graphs $\mathcal{X}_n$ and puzzles $\mathcal{P}_n(a_0)$}\mbox{}\\
Next we define the graph adapted for parameters which are \textbf{not} of Misiurewicz parabolic type. Let $a_0\in \partial \mathcal{C}_1\setminus\{0\}$ be such that $f^n_{a_0}(c_-(a_{0})) \not = 0,\forall n\geq 1$. Let $l\geq 2$, $\theta_l = \frac{\pm1}{2^l-1}$, then the internal ray $R^0_{a_0}(\theta_l)\subset B^*_{a_0}(0)$ is well defined and lands at a parabolic or repelling $l-$periodic point $x(a_0)$. One may choose $l$ large enough such that 
\begin{equation}\label{eq.land}
    v_-(a_0)\not\in f_{a_0}^{-n}(R^0_{a_0}(2^j\theta_l)),\,\,\forall 0\leq j\leq l-1,n\geq 0.
\end{equation}

In particular there is an external ray $R^{\infty}_{a_0}(\eta_l)$ with rational angle $\eta_l$ landing at $x(a_0)$. Now fix such an integer $l$, for each $n\geq0$ consider
\[\mathcal{X}_n = \bigcup_{k=0}^{n}\left(\bigcup_{\mathcal{U}_k\subset\mathcal{H}_k}\mathcal{E}_{\mathcal{U}_k}(n-k)\cup\bigcup_{\theta\in\Theta_{n-k}}\overline{\mathcal{R}_{\mathcal{U}_k}(\theta)}\right)\cup\left(\bigcup_{\eta\in H_n}\overline{\mathcal{R}_{\infty}(\eta)}\right)\cup\mathcal{E}_{\infty}(r^{1/3^n})\]
where
\[\Theta_n = \{\theta;\,\exists (i,j)\in[\![0,n]\!]\times[\![0,l-1]\!],2^{n-i}\theta =2^j \theta_l\}\]
\[H_n = \{\eta;\,\exists (i,j)\in[\![0,n]\!]\times[\![0,l-1]\!],3^{n-i}\eta =3^j \eta_l\}.\]
 We call a connected component $\mathcal{P}$ of $({\mathcal{W}(0)\cup S})\setminus\mathcal{X}_n$ a puzzle piece associated to $\mathcal{X}_n$ if it is bounded and $\partial \mathcal{P}\cap\mathcal{E}_{\infty}(r^{1/3^n})\not =\emptyset$. We denote it by $\mathcal{P}_n$ in the sequel. Clearly every $\mathcal{P}_{n+1}$ is contained in a unique $\mathcal{P}_n$. Define $\mathcal{P}_n(a_0)$ to be the puzzle piece containing $a_0$. (This is well-defined since $a_0\not\in \mathcal{X}_n$, for otherwise some external ray $R^{\infty}_{a_0}(\eta)$ or internal ray will land at a parabolic pre-periodic point or $v_-(a_0)$, contradicting with the construction of the dynamical graph for $f_{a_0}$). 

\begin{lemma}\label{lem.holomotion.internal}
For $n\geq0$, let $\mathcal{P}$ be the connected component of \[(\mathcal{W}(0)\cup S)\setminus\left(\bigcup_{k=0}^{n}\left(\bigcup_{\theta\in\Theta_{n-k}}\overline{\mathcal{R}_{\mathcal{U}_k}(\theta)})\right)\cup\Phi^{-1}(\overline{\Omega})\right)\]
containing $a_0$ and $\mathcal{P}'$ be the connected component $(\mathcal{W}(0)\cup S)\setminus\bigcup_{\eta\in H_n}\overline{\mathcal{R}_{\infty}(\eta)}$ containing $a_0$. Then there are dynamical holomorphic motions:
\[L_n: \mathcal{P}\times \bigcup_{j=0}^{n+1}f^{-j}_{a_0}(\bigcup_{i=0}^{l-1} R^0_{a_0}(2^i\theta_l))\longrightarrow \mathbb{C},\,\,L_n: \mathcal{P'}\times \bigcup_{j=0}^{n+1}f^{-j}_{a_0}(\bigcup_{i=0}^{l-1} R^\infty_{a_0}(3^i\eta_l))\longrightarrow \mathbb{C}.\]
\end{lemma}
\begin{proof}
This is a analogue version of Lemma \ref{lem.holomotion.parabolic}. We prove for the external rays and the proof for internal rays will be similar. By assumption, $a\not\in\bigcup_{\eta \in H_{n}} \overline{\mathcal{R}_{\infty}(\eta)}$ ensures that $\bigcup_{\eta \in H_{n+1}}\overline{R^\infty_{a}(\eta)}$ contains no critical point, therefore ${R}^{\infty}_{a}(\eta)$ will land at a parabolic or repelling pre-periodic point $x(a)$. There are only finitely many $a$ such that $x(a)$ is parabolic by the second step in the proof of Proposition \ref{propland}. So for all other $a\in \mathcal{P}'$, $x(a)$ is repelling. Take $a'$ such that $x(a')$ is parabolic. Then there exists $k\geq 0$ such that the holomorphic function $m(a) := ((f_a)^{k+l})'(x(a))$ satisfying $m(a')=1$ while for $a$ near $a'$, $|m(a)| \textgreater 1$, contradicting the fact that $m(a)$ is an open map. So for all $a\in \mathcal{P}'$, $x(a)$ is repelling. By Lemma \ref{lem.stab.misiur}, $L_n$ can be defined on $\mathcal{P}'$.
\end{proof}

\begin{corollary}\label{cor.para.Misiur}
Let $n\geq 1$. Any $a\in \mathcal{X}_n\cap\partial\mathcal{H}\cap \mathcal{P}_0(a_0)$ is a Misiurewicz parameter. In particular $a_0\not\in\mathcal{X}_n$ and $\mathcal{X}_n$ is connected.
\end{corollary}
\begin{proof}
By the lemma above $R^0_a(\theta_l)$ lands at a repelling periodic point. By Proposition \ref{prop.para.land.Misur} (or Remark \ref{rem.para.land.Misur}), $a$ is a Misiurewicz parameter. If $a_0\in\mathcal{X}_n$, then by Lemma \ref{lem.para.land.Misiur} there exists $\eta$ and $0\leq j\leq 1-1$ with $3^{n-j}\eta = \eta_l$ such that $R^{\infty}_{a_0}(\eta)$ lands at $v_-(a_0)$. But this contradicts (\ref{eq.land}). Again by Lemma \ref{lem.para.land.Misiur}, every landing point of $\mathcal{R}_{\mathcal{U}_k}(\theta)$ is the landing point of some $\mathcal{R}_{\infty}(\eta)$. ($\theta\in\Theta_{n-k},\eta\in H_n$). Hence $\mathcal{X}_n$ is connected.
\end{proof}

\subsection{Dynamical graphs and puzzles}
\paragraph{Dynamical graphs ${Y}^a_m$ and puzzles ${Q}^a_m$, $({Q}^a_m)^{\pm}$}\mbox{}\\
Let $\mathcal{Q}_n\not\subset\mathcal{W}(0)$ be a parameter puzzle piece of depth $n$. Following \cite{Roesch}, for every $a\in\mathcal{Q}_n$, $0\leq m\leq n+1$, one may construct the dynamical graph of depth $m$ adapted for points in the inverse orbit of 0:
 \[Y^a_0 := f_a(\partial\Omega_a)\cup R^{\infty}_{a}(0)\cup E^{\infty}_{a}(r),\,\,Y^a_m = f_a^{-m}(Y^a_0)\]
with $r\textgreater1$ the same as that in the para-graph $\mathcal{Y}_n$. For every $x\in f^{-m}_a(0)$, there is a unique external ray $R^{\infty}_a(t_x)$ converging to $a$ since $c_-(a)$ is not on $f_a^{-m}(R^{\infty}_a(0))$ (because $a\in \mathcal{Q}_n$). In particular $Y^a_m$ is connected and every connected component $Q$ of $\mathbb{C}\setminus Y^a_m$ is simply connected.

\begin{definition}
A \textbf{puzzle piece} $Q$ associated to $Y^a_m$ (denoted by $Q^m_a$) if it is a bounded a bounded connected component of $\mathbb{C}\setminus Y^a_m$ and its boundary contains parts of $E^{\infty}_a(r^{1/3^m})$. For every $1\leq m\leq n+1$ there are two adjacent puzzle pieces $(Q^m_a)^{\pm}$ (on the left,right hand side of $R^{\infty}_a(t_x)$ respectively) which commonly possess the segment of $R^{\infty}_a(t_x)$ between $x$ and $E^{\infty}_a(r^{1/3^m})$ as part of their boundaries.
\end{definition}

If $\mathcal{Q}_n\subset\mathcal{W}(0)$, similarly we define the dynamical graph adapted for points in the inverse orbit of $0$:
\[Y^a_0 := (f_a(\partial\Omega_a)\cup R^{\infty}_{a}(0)\cup R^{\infty}_{a}(\frac{1}{2})\cup E^{\infty}_{a}(r)),\,\,Y^a_m=f_a^{-m}(Y^a_0)\]
and the corresponding puzzle pieces $Q^a_m$ associated to it. In this case, let $x\in f_a^{-m}(0)$, there are exactly two external rays $R^{\infty}_a(t_x),R^{\infty}_a(t_x')$ converging to $x$ and let $(Q^m_a)^+$ (resp. $(Q^m_a)^-$) be the puzzle piece whose boundary contains the segment of $R^{\infty}_a(t_x)$ (resp. $R^{\infty}_a(t_x)$) between $x$ and $E^{\infty}_a(r^{1/3^m})$ but not that of $R^{\infty}_a(t_x')$ (resp. $R^{\infty}_a(t_x)$). Proposition 2.3 in \cite{Roesch} implies that in both cases, $\overline{(Q^a_m)^{\pm}}$ are shrinking to $x$, which can be restated as the theorem below:

\begin{theorem}\label{thm.pascale.parabolic}
Let $a\in \mathcal{C}_1\setminus \mathcal{H}$ and $f_a^n(c_-(a)) \not =0,\forall n\geq 1$. Take $x\in f_a^{-k}(0)$ ($k\geq 0$) and let $(Q^{a}_m)^{\pm}$ be as above. Then $\bigcap_{m\geq1} (Q^{a}_m)^{\pm} = \emptyset$. 
\end{theorem}

\paragraph{Dynamical graphs ${X}^a_m$ and puzzles ${P}^a_m$, ${P}^{a,v}_m$}\mbox{}\\
Next we construct the dynamical graph adapted for points in the Julia set which are not in the inverse orbit of 0. First recall the definition of renormalisable maps in the sense of Douady-Hubbard \cite{ASENS_1985_4_18_2_287_0}:
\begin{definition}
$f_a$ is \textbf{renormalizable} if there exists two simply connected open disks $U\subset\subset V$ and an integer $k\geq1$ such that $f_a^k:U\longrightarrow V$ is quadratic-like (i.e. holomorphic proper map of degree 2) and the orbit of the unique critical point of $f_a^k|_U$ stays in $U$. The \textbf{filled Julia set} $K(f^k_a|_U)$ is defined to be $\bigcap_n(f_a^k|_U)^{-n}(U)$.
\end{definition}

Let $a\in \mathcal{P}_n(a_0)$, $0\leq m\leq n+1 $, define the graph for $f_a$ of depth $m$ by
\[X^a_0 := f_a(\partial\Omega_a)\cup \left(\bigcup_{j=0}^{l-1}R^0_a(2^j\theta_l)\cup R^{\infty}_{a}(3^j\eta_l)\right)\cup E^{\infty}_{a}(r),\,\,X^a_m=f_a^{-m}(X^a_0).\]
with $r\textgreater 1$ and $l\geq 2$ being as in $\mathcal{X}_n$.

\begin{definition}
A \textbf{puzzle piece} $P$ associated to $X^a_m$ (denoted by $P^m_a$) if it is a bounded connected component of $\mathbb{C}\setminus X^a_m$ and its boundary contains parts of $E^{\infty}_a(r^{1/3^m})$. Denote by $P^{a,v}_{m}$ the \textbf{critical puzzle piece} of depth $m$, i.e. the puzzle piece containing the critical value $v_-(a)$.
\end{definition}

Proposition 5.3, Lemma 5.5 together with Lemma 5.6 in \cite{Roesch} gives 
\begin{theorem}\label{thm.dym}
Let $a\in \mathcal{C}_1\setminus (\mathcal{H}\cup\{0\})$ and $f^n_a(c_-(a))\not=0,\forall n\geq 1$. then there exists $l\textgreater 1$, $\theta_l = \frac{\pm1}{2^l-1}$ and a sequence of non-degenerated annuli $A^a_{n_i},{i\geq 0}$, such that 

\begin{enumerate}
    \item $A^a_{n_i} = P^{a,v}_{n_i}\setminus\overline{P^{a,v}_{n_i+1}}$ for $i\geq 1$, 
    \item $f^{n_i-n_0}_a:A^a_{n_i}\longrightarrow A^a_{n_0}$ is a non-ramified covering.
    \item if $a\in\mathcal{W}(0)$, then $\sum_i mod(A^a_{n_i}) = \infty$. 
    \item\label{dichotomie} if $a\not\in\mathcal{W}(0)$, then either $\sum_i mod(A^a_{n_i}) = \infty$ or there exists $k\geq 1$ such that $f_a^k: P^{a,v}_{m+k}\longrightarrow P^{a,v}_m$ is quadratic-like for all $m$ large enough and $\bigcap P^{a,v}_m$ is the filled Julia set of the renormalized map $f_a^k$.
\end{enumerate}
\end{theorem}

\begin{remark}
Notice in the second alternative in \ref{dichotomie} of Theorem \ref{thm.dym}, $f_a$ is renormalisable since $v_-(a)\in\bigcap_n P^{a,v}_n$.
\end{remark}

\begin{definition}\label{def.renor}
A non Misurewicz parabolic parameter $a\in a\in \mathcal{C}_1\setminus (\mathcal{H}\cup\{0\})$ is \textbf{renormalisable} if $a$ satisfies the second alternative in \ref{dichotomie} of Theorem \ref{thm.dym}.
\end{definition} 
\begin{remark}
In fact one may adapt the proof of Lemma 3.24 in \cite{Roesch} to show that the parameter $a$ is renormalisable if and only if the corresponding polynomial $f_a$ is renormalisable. 
\end{remark}

\section{Local connectivity of  $\partial\mathcal{H}_n$}\label{sec.loc}

\subsection{Misiurewicz parabolic case}
Concluding from Lemma \ref{lem.holomotion.external.parabo} and Lemma \ref{lem.holomotion.equi} we get
\begin{lemma}\label{lem.holomotion.parabolic}
Let $a_0\in \mathcal{Q}_n$. Then for $0\leq m\leq n+1$ there exists a dynamical holomorphic motion $L_m:\mathcal{Q}_n\times Y^{a_0}_m\longrightarrow \mathbb{C}$ with base point $a_0$ such that $L_m(a,Y^{a_0}_m) = Y^a_m$.
\end{lemma}

\begin{proposition}\label{prop.loc.paramisiur}
Let $\mathcal{U}\subset\mathcal{H}_k$ ($k\geq 0$) be a connected component. Then $\partial\mathcal{U}$ is locally connected at all $a_0\in\mathcal{U}$ such that $f_{a_0}^m(c_-(a_0)) = 0$ for some $m\geq 1$.
\end{proposition}
\begin{proof}
Let $a_0\in \overline{\mathcal{W}(0)\cup S}$ satisfy the hypothesis. First suppose that $a_0\not=2$. Consider the two following cases respectively:
\begin{itemize}
  \item $a_0\not\in\mathcal{W}(0)$. By Lemma \ref{lemparaland}, let $t_0$ be the angle such that $\mathcal{R}_{\infty}(t_0)$ lands at $a_0$. Then $\mathcal{R}_{\infty}(t_0)$ is part of $\mathcal{Y}_n$ for all $n\geq N$, where $N$ is some fixed integer. Since $t_0\not=0$, there are exactly two adjacent puzzle pieces $\mathcal{Q}^{+}_n,\mathcal{Q}^{-}_n$ which commonly possess $\overline{\mathcal{R}_{\infty}(t_0)}\cap\{a;\,\,|\Phi_{\infty}(a)|\leq r^{1/3^n}\}$ as part of their boundaries, where $\mathcal{Q}^{+}_n$ is defined to be the one intersecting external rays with angles larger than $t_0$. Clearly $\mathcal{Q}^{\pm}_{n+1}\subset\mathcal{Q}^{\pm}_{n}$. Now we claim that for $a\in \mathcal{Q}^{\pm}_{n}$, $v_-(a)\in (Q^a_n)^{\pm}$. Indeed, for any $\epsilon\textgreater 0$, there exists $n'\textgreater n$ large enough and $t'_0$ such that $3^{n'}t'_0 = 1$, $\mathcal{R}_{\infty}(t'_0)\cap \mathcal{Q}^{\pm}_n \not=\emptyset$ and $|t'_0-t_0|\textless \epsilon$. Hence for $a'\in \mathcal{R}_{\infty}(t'_0)\cap \mathcal{Q}^{\pm}_n$, $v_-(a')\in (Q^{a'}_n)^{\pm}$. By Lemma \ref{lem.holomotion.parabolic}, $Y^n_a$ moves holomorphically for $a\in \mathcal{Q}^{\pm}_n$. Notice that $v_-(a)$ also moves holomorphically and it does not belong to $Y^n_a$ when $a\in \mathcal{Q}^{\pm}_n$, we deduce that $v_-(a)\in (Q^a_n)^{\pm}$ since $v_-(a')\in (Q^{a'}_n)^{\pm}$. Now if $a\in \bigcap_n\mathcal{Q}_n^{\pm}$, then $a\in \mathcal{C}_1\cap\mathcal{H}^c$ since $a$ is excluded by all equipotentials in $\mathcal{H}_{\infty}$ and $\mathcal{H}$. Moreover $f^n_a(c_-(a)) \not= 0,\forall n\geq 1$ by definition of $\mathcal{Q}_n$. Therefore by Theorem \ref{thm.pascale.parabolic}, $\bigcap_{m\geq1} (Q^{a}_m)^{\pm} = \emptyset$. But on the other hand $v_-(a)\in \bigcap_n(Q^a_n)^{\pm}$, a contradiction. Hence $\bigcap_n\mathcal{Q}_n^{\pm} = \emptyset$. This means that if we set 
  \[\mathcal{O}_n := \mathcal{Q}_n^{+}\cup\mathcal{Q}_n^{-}\cup({\mathcal{R}_{\infty}(t_0)}\cap\{a;\,\,|\Phi_{\infty}(a)|\leq r^{1/3^n}\})\cup\{a_0\}.\]
  then $(\mathcal{O}_n\cap\partial\mathcal{U})_n$
  form a basis of connected neighborhood of $a_0$.
  \item $a_0\in \mathcal{W}(0)$. The proof is similar as above. The only difference is that here we use the parameter puzzles in $\mathcal{W}(0)$ and the corresponding dynamical puzzles.
\end{itemize}
Finally if $a_0=2$, set $\mathcal{Q}_n$ to be the puzzle piece containing ${\mathcal{R}_{\infty}(t_0)}\cap\{a;\,\,|\Phi_{\infty}(a)|\leq r^{1/3^n}\}$ in its boundary. Define the basis of connected neighborhood of $2$ by setting ($\tau:z\mapsto\overline{z}$)
\[\mathcal{O}_n := \mathcal{Q}_n\cup\tau\mathcal{Q}_n\cup({\mathcal{R}_{\infty}(t_0)}\cap\{a;\,\,|\Phi_{\infty}(a)|\leq r^{1/3^n}\})\cup\{2\}. \]
\end{proof}

\begin{figure}[H] 
\centering 
\includegraphics[width=0.8\textwidth]{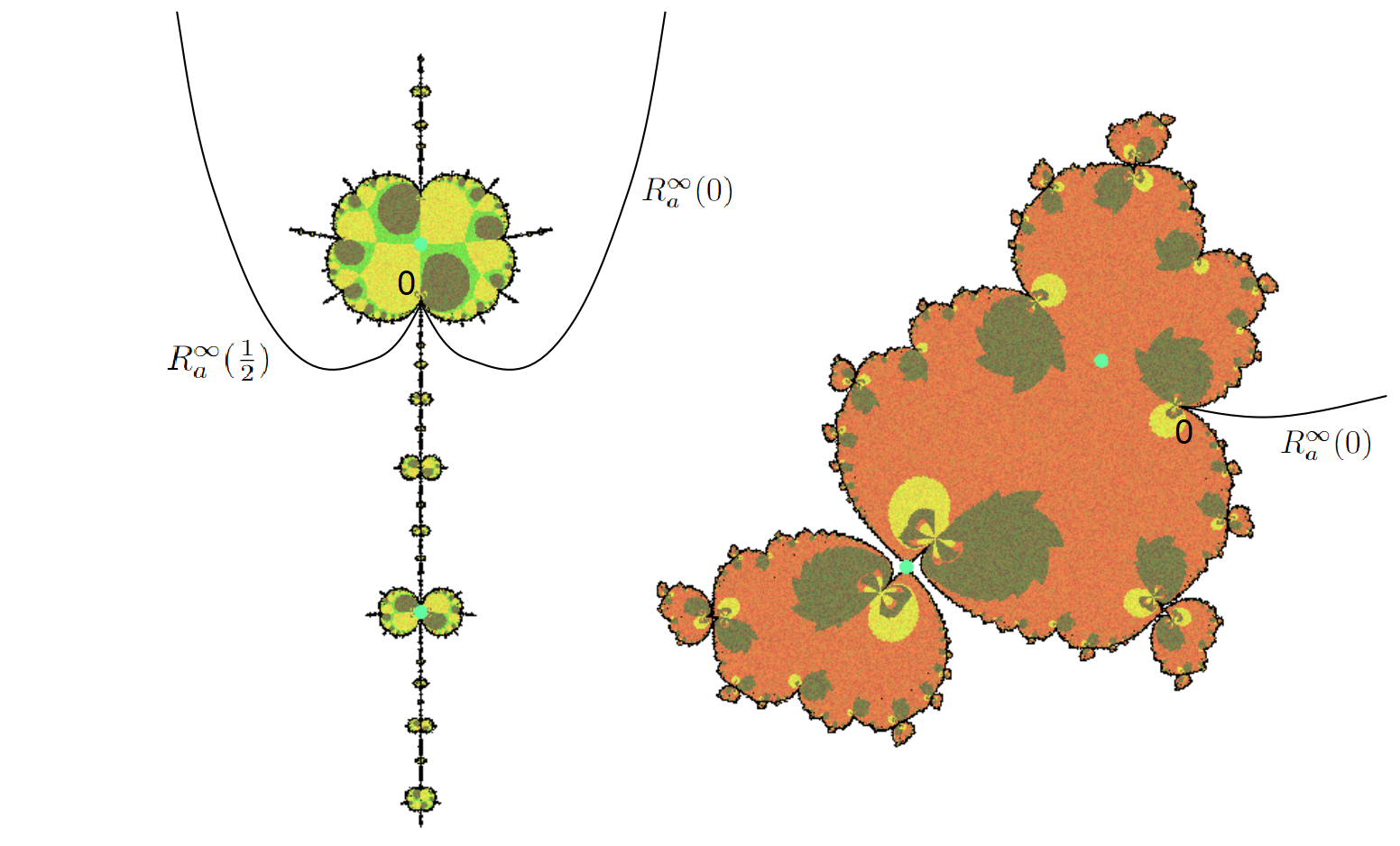} 
\caption{The Julia sets of two different types of Misiurewicz parabolic parameters: the one on the left is the Julia set of the landing point of $\mathcal{R}^S_{\infty}(\frac{2}{3}),\mathcal{R}^{iS}_{\infty}(\frac{5}{6})$ (which is contained in $\mathcal{W}(0)$) and the right one is that of the landing point of $\mathcal{R}^S_{\infty}(\frac{1}{3})$ (which is contained in $S\setminus\overline{\mathcal{W}(0)}$).} 
\label{Fig.main2} 
\end{figure}

\begin{remark}
Notice that $a_0 = 0$ is not included in Proposition \ref{prop.loc.paramisiur}. But still we can verify the local connectivity with a very similar argument: let $\mathcal{Q}_n\not\subset\mathcal{W}(0)$ be the unique puzzle piece of depth $n$ containing $0$ on its boundary. For $a\in\mathcal{Q}_n$, the critical value $v_-(a)$ is contained in the connected component bounded by $R^{\infty}_a(0),R^{\infty}_a(\frac{3^n-1}{2\cdot3^n})$ and a segment of equipotential in $B_a^*(0)$ linking them. If $\bigcap{\mathcal{Q}_n}\not = \emptyset$, then for $a\in\bigcap{\mathcal{Q}_n}$, $a$ is not Misiurewicz parabolic and the puzzle pieces $(Q^a_n)^{\pm}$ adjacent at $R^{\infty}_a(0)$ are well defined for all depth $n$. Moreover, since the graphe $Y^b_n$ moves holomorphically for $b\in\mathcal{Q}_n$, we deduce that $v_-(a)\in (Q^a_n)^-$. This contradicts Theorem \ref{thm.pascale.parabolic} which tells us that $\bigcap_{n\geq1}(Q^a_n)^-=\emptyset$.
\end{remark}

\subsection{Non renormalizable case}
In this subsection, we always fix some $a_0\in\partial\mathcal{U}\setminus 0$ which is \textbf{not} Misiurewicz parabolic and consider the corresponding graph $\mathcal{X}_n$ and puzzle pieces $\mathcal{P}_n(a_0)$.
Concluding from Lemma \ref{lem.holomotion.equi} and Lemma \ref{lem.holomotion.internal} we get 
\begin{lemma}\label{lem.holomotion.misur}
For $0\leq m\leq n+1$, there is a dynamical holomorphic motion $L_m:\mathcal{P}_n(a_0)\times X^{a_0}_m\longrightarrow \mathbb{C}$ with base point $a_0$ such that $L_m(a,X^{a_0}_m) = X^a_m$.
\end{lemma}

Let us first make an important remark on the motion of $\partial P^{a_0,v}_n$ ($n\geq 1$): if $a\in \mathcal{P}_n(a_0)$, then $L_n(a,\partial P^{a_0,v}_n)$ is the boundary of the puzzle piece of $f_a$ containing $v_-(a)$. However if $a\in \mathcal{P}_{n-1}(a_0)$, then $L_{n}(a,\partial P^{a_0,v}_n)$ does not necessarily bound $v_-(a)$ (for example, $v_-(a)\in X^n_{a}$ when $a\in \mathcal{X}_n\cap\mathcal{P}_{n-1}(a_0)$).

Next we give the relation between para-puzzles and dynamical puzzles:
\begin{lemma}\label{lem.homeo.para.dym}
Let $n\geq 0$. The mapping $H_n:\mathcal{P}_n(a_0)\cap\mathcal{X}_{n+1}\longrightarrow P^{a_0,v}_{n}\cap X^{n+1}_{a_0}$ defined by $H_n(a) = (L^a_{n+1})^{-1}(v_-(a))$ is injective. Moreover there exists $N\geq 0$ such that $\forall n\geq N$, $H_n$ is surjective.
\end{lemma}
\begin{proof}
First we prove that $H_n$ is well-defined. By definition, we have $v_-(a)\in P_n^{a,v}\cap X^{n+1}_{a}$. Consider the holomorphic motion starting at $a$: $\Tilde{L}_{n+1}(a',z) := L_{n+1}(a',(L^a_{n+1})^{-1}(z))$. By continuity there exists a disk $B(a,r)$ on which $v_-(a')$ is surrounded by $\Tilde{L}_{n+1}(a',\partial P_n^{a,v})$. Pick any $a'$ on $\partial B(a,r)$, if $a'\not\in \partial\mathcal{P}_n(a_0)$, then for any $a''$ in a small disk $B(a',r')$, $f_{a''}(c_-(a''))$ is surrounded by $\Tilde{L}_{n+1}(a'',\partial P_n^{a,v})$. Hence we can extend this property step by step until we reach $\partial \mathcal{P}_n(a_0)$. In particular $H_n(a)\in P^{a_0,v}_{n}$. By Lemma \ref{lem.holomotion.misur}, $\partial P_{n+1}^{a_0,v}$ moves holomorphically when $a\in \mathcal{P}_n(a_0)$, hence $H_n(a)\in X^{n+1}_{a_0}$.

Next we verify injectivity. The injectivity is clear for external rays and equipotentials since angles and equipotentials are preserved by $L_{n+1}$. We will only prove for internal rays. The proof for internal equipotentials are similar. Suppose there are two distinct parameters $a,a'$ such that $H_n(a) = H_n(a')$. Then clearly $a,a'$ belong to different connected components $\mathcal{U},\mathcal{U}'$. So set $a\in \mathcal{R}_\mathcal{U}(t),a'\in\mathcal{R}_\mathcal{U'}(t')$. 
\begin{itemize}
    \item First suppose that the landing point $b,b'$ of $\mathcal{R}_\mathcal{U}(t),\mathcal{R}_\mathcal{U'}(t')$ do not coincide. Consider the external rays $\mathcal{R}_{\infty}(s),\mathcal{R}_{\infty}(s')$ landing at $b,b'$ respectively. Then the two rays $H_n(\mathcal{R}_{\infty}(s)),H_n(\mathcal{R}_{\infty}(s'))$ land at a common point $x(a_0)$ since $H_n(a) = H_n(a')$. But this is impossible since $f_{a_0}$ is injective near the forward orbit of $x(a_0)$.
    \item Next if $b = b'$. Then there exists $\mathcal{R}_{\mathcal{U}}(\Tilde{t}),\mathcal{R}_{\mathcal{U}'}(\Tilde{t}')$ landing at $\Tilde{b},\Tilde{b}'$ respectively which have the same image under $H_n$. Then $\Tilde{b}\not = \Tilde{b}'$, for otherwise one can find a loop in $\overline{\mathcal{U}\cup\mathcal{U}'}$ surrounding points in $\mathcal{H}_{\infty}$. We repeat the argument above to $\Tilde{b},\Tilde{b}'$.
\end{itemize}

Finally we verify surjectivity. Since $a_0$ is not Misiurewicz parabolic, there exists $N\geq0$ such that $\forall n\geq N$, $\overline{P^{a_0,v}_n}\cap R^{\infty}_a(0)=\emptyset$. Therefore for $n\geq N$, $H_n$ is surjective on the part of external equipotential and external rays. So we verify surjectivity on the part of internal rays and the case of internal equipotentials are similiar. Let $z_0\in X^{n+1}_{a_0}\cap B_{a_0}(0)\cap P^{a_0,v}_n$. We may suppose that $z_0$ is on some connected component $\zeta$ of $f^{-k}_{a_0}(R^0_{a_0}(2^j\theta_l))$. Let $x_0$ be the landing point of $\zeta$, then there is an external ray $R^{\infty}_{a_0}(t)$ which belongs to $X^{n+1}_{a_0}$ landing at $x_0$. By corollary \ref{cor.para.Misiur}, $\mathcal{R}_{\infty}(t)$ lands at a Misiurewicz parameter $a'$ which is accessible by some $\mathcal{R}_{\mathcal{U}_k}(\theta)$. By continuity we have $H_n(\mathcal{R}_{\mathcal{U}_k}(\theta))\subset \zeta$. By construction of the parameter graph, $\mathcal{R}_{\mathcal{U}_k}(\theta)$ will hit some internal equipotential, and since $H_n$ preserves equipotentials and rays, the $\subset$ is actually $=$.
\end{proof}

\begin{corollary}\label{cor.criticalpuzzle}
Let $a\in \mathcal{P}_{n-1}(a_0)$. Let $C^a_n$ be the puzzle piece bounded by $L_n(\partial P^{a_0}_n)$. Then $a\not\in \overline{\mathcal{P}_{n}(a_0)}$ if and only if $v_-(a)\not\in \overline{C^a_n}$.
\end{corollary}
\begin{proof}
If $a\not\in \overline{\mathcal{P}_{n}(a_0)}$, then take a simple path $a_t\subset \mathcal{P}_{n-1}(a_0)$ connecting $a_0,a$ such that $a_t\cap\mathcal{P}_{n}(a_0)$ only contains one point. Thus Lemma \ref{lem.homeo.para.dym} ensures that once $a_t$ goes out of $\overline{\mathcal{P}_{n}(a_0)}$, $f_{a_t}(c_-(a_t))$ will never enter again $\overline{C^a_n}$. Conversely if $a\in\mathcal{P}_n(a_0)$, then clearly $v_-(a)\in C^a_n$ since $\partial P^{a_0,v}_n$ moves holomorphically and $v_-(a)$ does not intersect $X^a_n$.
\end{proof}

\begin{corollary}
For $n$ large enough, if $\overline{P^{a_0}_{n+1}}\subset P^{a_0}_n$, then $\overline{\mathcal{P}_{n+1}(a_0)}\subset\mathcal{P}_n(a_0)$.
\end{corollary}
\begin{proof}
By Lemma \ref{lem.homeo.para.dym} $H^{-1}_n(\overline{\partial P^{a_0}_{n+1}})\subset \mathcal{P}_n(a_0)$ bounds a puzzle piece $\mathcal{P}$. So it suffices to prove that $a_0\in\mathcal{P}$. Suppose not, then there is $a'\in\partial\mathcal{P}$ but $a'\not \in \overline{\mathcal{P}_{n+1}(a_0)}$. By Corollary \ref{cor.criticalpuzzle}, $v_-(a')\not\in\overline{C^{a'}_{n+1}}$, a contradiction.
\end{proof}

By Theorem \ref{thm.dym}, there exists $\theta_l = \frac{\pm 1}{2^l-1}$ with $l$ large enough, such that in the dynamical plane of $f_{a_0}$ one has a sequence of non-degenerated annuli $A^{a_0}_{n_i} := P^{a_0,v}_{n_i}\setminus\overline{P^{a_0,v}_{n_i+1}}$. The above corollary implies that in the parameter plane, $\mathcal{A}^{a_0}_{n_i} := \mathcal{P}_{n_i}(a_0)\setminus\overline{\mathcal{P}_{n_i+1}}(a_0)$ is also non-degenerated for $i$ large enough. Applying Shishikura's trick, we can get the estimate between the moduli of the para-annuli and dynamical annuli (see also \cite{Roesch1} Section 4.1):

\begin{lemma}\label{lem.shishikura}
There exists $K\textgreater 1$ such that for $i$ large enough
\[\frac{1}{K}mod(A^{a_0}_{n_i})\leq mod\mathcal(\mathcal{A}_{n_i})\leq Kmod(A^{a_0}_{n_i}).\]
\end{lemma}
\begin{proof}
Consider the sequence of annuli $(A^{a_0}_{n_i})_{i\geq 0}$ in Theorem \ref{thm.dym} for $a_0$. For $i=0$ we set $A^{a_0}_{n_0} = P^{a_0}_{n_0}\setminus \overline{P^{a_0}_{n_0+1}}$. Now fix any $n=n_i$, by Lemma \ref{lem.holomotion.misur}, we have below the commutative diagram on the left for $a\in\mathcal{P}_n(a_0)$, where $A^{a}_{n}$ is just the annulus bounded by $L^a_n(\partial A^{a_0}_n)$. Applying Słodkowski's extension theorem \cite{Sl}, the holomorphic motion $L_{n_0}:\mathcal{P}_{n_0}(a_0)\times\partial{A}^{a_0}_{n_0}\longrightarrow \mathbb{C}$ can be quasiconformally extended to $\Tilde{L}_{n_0}:\mathcal{P}_{n_0}(a_0)\times\mathbb{C}\longrightarrow \mathbb{C}$ and we lift the extended mapping $\Tilde{L}^a_{n_0}$ by $f^{n-n_0}_{a_0},f^{n-n_0}_{a}$ to get the commutative diagram on the right:
\[\xymatrix{
 \partial A^{a_0}_n\ar[r]^{L^a_n}\ar[d]_{f^{n-n_0}_{a_0}} &   \partial A^{a}_{n} \ar[d]_{f^{n-n_0}_{a}}\\
\partial A^{a_0}_{n_0} \ar[r]^{L^a_{n_0}} &  \partial A^{a}_{n_0}
}
\,\,\,\,\,\,\,
\xymatrix{
  \overline{A^{a_0}_n} \ar@{.>}[r]^{\Tilde{L}^a_n}\ar[d]^{f^{n-n_0}_{a_0}} &    \overline{A^{a}_{n}} \ar[d]^{f^{n-n_0}_{a}}\\
 \overline{A^{a_0}_{n_0}} \ar[r]^{\Tilde{L}^a_{n_0}} &   \overline{A^{a}_{n_0}}
}
\]
In particular, the dilatation of $\Tilde{L}^{a}_n$ equals that of $K_a := \Tilde{L}^{a}_{n_0}$. Define $\Tilde{H}_n:\mathcal{A}^{a_0}_n\longrightarrow {A}^{a_0}_n$ by $a\mapsto (\Tilde{L}^{a}_n)^{-1}(v_-(a))$. Notice that $\Tilde{H}_n$ is well-defined by Corollary \ref{cor.criticalpuzzle}. Hence we have $\Tilde{L}^{a}_n(\Tilde{H}_n(a)) = v_-(a)$. One may verify easily from this that $\Tilde{H}_n$ is locally invertible except at finitely many points. Differentiating with $\partial_{\overline{a}}$ on both sides gives that 
\[\partial_{\overline{a}} \Tilde{H}_n(a) = -\frac{\partial_{\overline{z}}\Tilde{L}^{a}_n}{\partial_{z}\Tilde{L}^{a}_n}(\Tilde{H}_n(a))\cdot\overline{\partial_a\Tilde{H}_n(a)}.
\]
Hence the dilatation of $\Tilde{H}_n$ is bounded by $\sup_a\{K_a;\,a\in \mathcal{P}_{n_0+1}(a_0)\}\textless+\infty$. Recall that $\Tilde{H}_n$ is almost everywhere locally invertible and therefore it is injective since its extension on $\partial\mathcal{P}_{n+1}(a_0)$ is injective. The surjectivity is clear since it is a proper map. Therefore $\Tilde{H}_n$ is a $K-$quasiconformal homeomorphism with $K$ not depending on $n$. The result then follows.
\end{proof}

\begin{corollary}
Let $a_0\in\partial\mathcal{U}\setminus\{0\}$. If $a_0\not\in \mathcal{W}(0)$ satisfies the first alternative in Theorem \ref{thm.dym} or $a_0\in\mathcal{W}(0)$, then $\partial\mathcal{U}$ is locally connected at $a_0$. 
\end{corollary}
\begin{proof}
This follows immediately from  the above lemma and Grötzsch's inequality.
\end{proof}

\subsection{Renormalizable case}\label{subsec.renor}
In this subsection we always fix a non Misiurewicz parabolic parameter $a_0\in\partial\mathcal{U}\setminus (\{0\}\cup\mathcal{W}(0))$ which is renormalisable. Recall that a set $\textbf{M}'\subset \mathbb{C}$ is called a copy of the Mandelbrot set $\textbf{M}$ if there exists $k\geq 1$ and a homeomorphism $\chi:\textbf{M}'\longrightarrow\textbf{M}$ such that $f_a$ is $k-$renormalisable and $f^k_a$ is quasiconformally conjugated to $z^2+\chi(a)$.

\begin{figure}[H] 
\centering 
\includegraphics[width=0.6\textwidth]{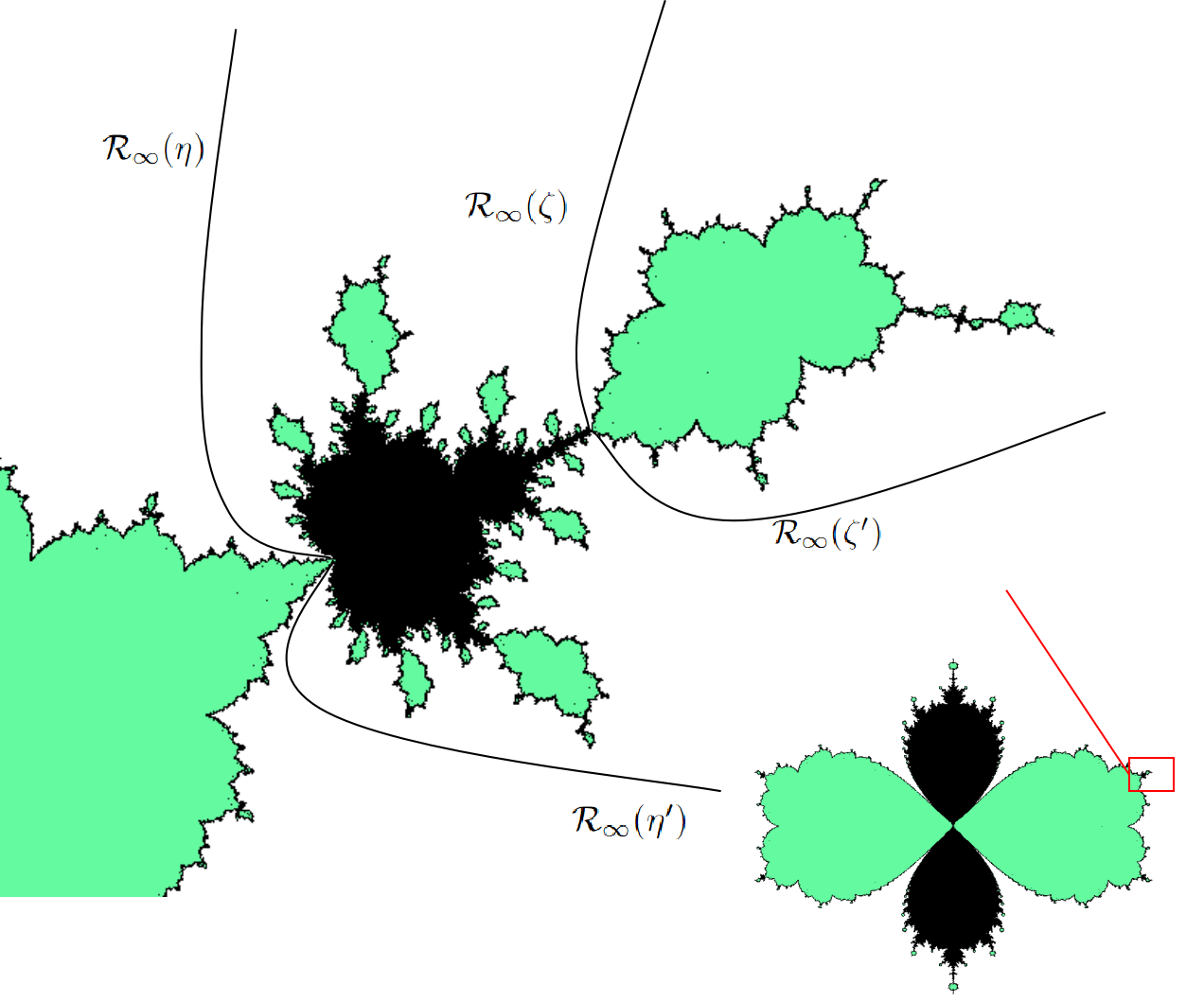} 
\caption{A zoom of a copy of Mandelbrot set $\textbf{M}_{a_0}$ attached at $\partial\mathcal{U}_0$. Here the period of renormalisation is 2. The external rays $\mathcal{R}_{\infty}(\eta),\mathcal{R}_{\infty}(\eta')$ land at the cusp of $\textbf{M}_{a_0}$, separating $\textbf{M}_{a_0}$ with $\mathcal{U}_0$. The external rays $\mathcal{R}_{\infty}(\zeta),\mathcal{R}_{\infty}(\zeta')$ land at a Misiurewicz parabolic parameter, separating a component of $\mathcal{H}_2$ with $\textbf{M}_{a_0}$.} 
\label{fig.para.renor} 
\end{figure}

\begin{figure}[H]\label{fig.dym.renor} 
\centering 
\includegraphics[width=0.5\textwidth]{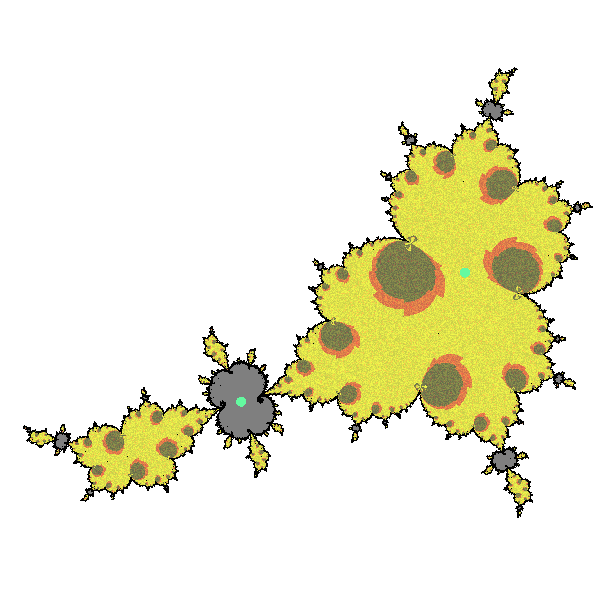} 
\caption{The Julia set $f_{a_0}$ with $a_0$ the cusp of $\textbf{M}_{a_0}$ in Figure \ref{fig.para.renor}. There is a quasiconformal copy of cauliflower (grey) attached at $B^*_{a_0}(0)$.} 
\label{Fig.main2} 
\end{figure}

\begin{proposition}\label{prop.mandel}
$\textbf{M}_{a_0} := \bigcap\mathcal{P}_n(a_0)$ is a copy of the Mandelbrot set. Moreover $\textbf{M}_{a_0}\cap\partial\mathcal{U} = \{a_0\}$.
\end{proposition}
\begin{proof}
First Notice that $\bigcap\mathcal{P}_n(a_0) = \bigcap\overline{\mathcal{P}_{n}(a_0)}$ since $\overline{\mathcal{P}_{n_i+1}(a_0)}\subset \mathcal{P}_{n_i}(a_0)$, where the sequence $(n_i)$ is as in Theorem \ref{thm.dym}. By Theorem \ref{thm.dym}, there exists $N$ such that for $n\geq N$ and $a\in \mathcal{P}_n(a_0)$, $f^k_a(P^{a,v}_n) = P^{a,v}_{n-k}$, where $P^{a,v}_m := L^a_m(\partial P^{a_0,v}_m), m\leq n$ is the puzzle piece containing $v_-(a)$. Consider the analytic family of quadratic maps $\{f^k_a:P^{a,v}_N\longrightarrow P^{a,v}_{N-k},a\in \mathcal{P}_N(a_0)\}$. Set $\textbf{M}_{\textbf{f}} = \{a;\,K(f^k_a|_{P^{a,v}_N}) \text{ is connected}\}$. It is easy to see that $\textbf{M}_{\textbf{f}} = \textbf{M}_{a_0}$. Indeed, if $K(f^k_a|_{P^{a,v}_N})$ is connected but $a\not\in \overline{\mathcal{P}_n(a_0)}$ for some $n$, then by Corollary \ref{cor.criticalpuzzle} $v_-(a)\not\in \overline{C^a_n}$, hence it will eventually escape $P^{a,v}_N$ under iteration of $f^k_a$, a contradiction; conversely, if $a\in \bigcap_n\mathcal{P}_n(a_0)$, then $v_-(a)$ will never escape $P^{a,v}_N$. Thus the Mandelbrot-like families theory of Douady-Hubbard \cite{ASENS_1985_4_18_2_287_0} gives a continuous map $\chi: \mathcal{P}_N(a_0)\longrightarrow \mathbb{C}$ such that $f^k_a$ is quasiconformally conjugated to $z^2+\chi(a)$ and $\chi^{-1}(\textbf{M}) = \textbf{M}_{a_0}$. Take $\mathcal{P}_p(a_0)$ compactly contained in $\mathcal{P}_N(a_0)$ with $\partial\mathcal{P}_p(a_0)$ homeomorphic to a circle. We claim that when $a$ describes $\partial\mathcal{P}_p(a_0)$ once, the curve $f_a^k(w_a)-w_a$ will also describe 0 once, where $w_a$ is the unique critical point of $f_a^k(z)$ in ${P^{a,v}_p}$. In fact, if we regard $\partial\mathcal{P}_p(a_0)$ as the unit circle, then $f_a^k(w_a)-w_a$ is homotopy to $H_{p-1}(a)-w_{a_0}$ by $G(t,a) := L^{a+t(a-a_0)}_{p}(H_{p-1}(a))-w(a+t(a-a_0))$ (notice that $f_a^k(w_a)-w_a = v_-(a)-w_a$). Since $H_{p-1}$ is a homeomorphism from $\partial\mathcal{P}_p(a_0)$ to $\partial P^{a_0,v}_p$, the winding number of $H_{p-1}(a)-w_{a_0}$ is exactly 1. Thus from \cite{ASENS_1985_4_18_2_287_0} we deduce that $\chi$ is actually a homeomorphism.

Finally we prove that $\textbf{M}_{a_0}\cap\partial\mathcal{U} = \{a_0\}$. 
\begin{itemize}
    \item $\mathcal{U}$ is of adjacent type. For $n$ large enough, $\partial\mathcal{P}_n(a_0)$ is a topological circle and consists of two external rays $\mathcal{R}_{\infty}(\eta_n),\mathcal{R}_{\infty}(\eta'_n)$ with $\eta_n\textgreater\eta_n'$ landing at $\partial\mathcal{U}$. It is easy to see that $\eta_n,\eta_n'$ correspond respectively to the biggest and smallest angle of the external rays involved in $\partial \mathcal{U}$. Moreover $\eta_n$ is decreasing and $\eta_n'$ is increasing, hence they converge to some $\eta,\eta'$ respectively. Recall that in Lemma \ref{lem.homeo.para.dym} $H_{n-1}$ gives a homeomorphism between $\partial\mathcal{P}_n(a_0)$ and $\partial{P}^{a_0,v}_n$ preserving the angles of external rays, and since $a_0$ is $k-$renormalisable, we deduce that $3^k\eta_{n+k} = \eta_n$, $3^k\eta'_{n+k} = \eta'_n$. Taking limit in $n$, one sees that $\eta,\eta'$ are both of period $k$. Let $b$ (resp.$b'$) be the landing point of $\mathcal{R}_{\infty}(\eta)$ (resp.$\mathcal{R}_{\infty}(\eta')$). Then $b\in\textbf{M}_{a_0}$ since $\mathcal{R}_{\infty}(\eta)$ intersects every $\mathcal{P}_n(a_0)$. Therefore $R^{\infty}_b(\eta)$ lands at a parabolic $k-$periodic point with multiplier $1$. Thus $\chi(b) = \frac{1}{4}$. For the same reason $\chi(b') = \frac{1}{4}$, hence $b = b' = a_0$.
    \item $\mathcal{U}$ is of capture type. In the dynamical plan of $f_{a_0}$, there exists $i\geq 1$ such that $f^i_{a_0}(P^{a_0,v}_n)$ intersects $B^*_{a_0}(0)$. Since $f_{a_0}$ is renormalisable, $P^{a_0,v}_n$ will intersect $B^*_{a_0}(0)$ for all $n$ large enough. Transfering by $H_n$, $\mathcal{P}_n(a_0)$ intersects the adjacent component $\mathcal{U}_0$. Thus $(\bigcap_n\mathcal{P}_n(a_0))\cap\partial\mathcal{U}_0\not=\emptyset$. By the first step, the intersection is actually a point $a_1\in\partial\mathcal{U}_0$. In particular $\mathcal{P}_n(a_0) = \mathcal{P}_n(a_1)$. Let $\zeta_n\textgreater\zeta_n'$ be the angles such that in $\partial\mathcal{P}_n(a_0)$, $\mathcal{R}_{\infty}(\zeta_n),\mathcal{R}_{\infty}(\zeta_n')$ land at $\partial\mathcal{U}$ and they converge to $\zeta,\zeta'$ respectively. By transferring $\partial\mathcal{P}_{n}(a_0)$ to $\partial P^{a_0,v}_n$ with $H_{n-1}$ one can prove that $d^i\zeta_{n+i} = \eta_n,d^i\zeta'_{n+i} = \eta'_n$. Taking limit in $n$, we get $d^i\zeta= \eta,d^i\zeta'= \eta'$. 
    
    Let $c,c'$ be the landing point of $\mathcal{R}_{\infty}(\zeta),\mathcal{R}_{\infty}(\zeta')$. Clearly $c,c'\in\mathcal{W}(a_1)$, where $\mathcal{W}(a_1)$ is the open region bounded by $\mathcal{R}_{\infty}(\eta),\mathcal{R}_{\infty}(\eta')$. By a standard holomorphic motion argument one can show that $\forall a\in \mathcal{W}(a_1)$, $R^{\infty}_{a}(\eta),R^{\infty}_{a}(\eta')$ land at the same repelling $k-$periodic point. Therefore if one of the two rays $R^{\infty}_{c}(\zeta),R^{\infty}_{c}(\zeta')$ lands at a repelling pre-periodic point, then so does the other one. Thus by Remark \ref{rem.para.land.Misur} and Lemma \ref{lem.para.land.Misiur} we see that $c = c'$. If both $R^{\infty}_{c}(\zeta),R^{\infty}_{c}(\zeta')$ land at a parabolic pre-periodic point, then $\chi(c) = \chi(c') = \frac{1}{4}$ and hence $c = c' = a_1$. In both cases, $\mathcal{R}_{\infty}(\zeta),\mathcal{R}_{\infty}(\zeta')$ seperate $\textbf{M}_{a_0}$ and $\mathcal{U}$. Therefore $(\bigcap_n\mathcal{P}_n(a_0))\cap\partial\mathcal{U} = \{a_0\}$.
\end{itemize}

\end{proof}

\begin{corollary}
Let $a_0\in(\partial\mathcal{U}\setminus\{0\})\cap\mathcal{W}^c(0)$ satisfy the second alternative in Theorem \ref{thm.dym} then $\partial\mathcal{U}$ is locally connected at $a_0$. 
\end{corollary}

\begin{proof}
By the above proposition $(\mathcal{P}_n(a_0)\cap\partial\mathcal{U})_n$ form a connected basis of $a_0$. 
\end{proof}

\begin{proof}[\textbf{Proof of Theorem~\ref{thm.A}}]
Let $\mathcal{U}$ be a connected component of $\mathcal{H}$. Since $\partial\mathcal{U}$ is locally connected, then any conformal representation $\Psi:\mathbb{D}\longrightarrow\mathcal{U}$ can be extended continuously and surjectively to the boundary. Moreover $\Psi:\partial\mathbb{D}\longrightarrow\partial\mathcal{U}$ is injective: if not, then there exists $a\in\partial\mathcal{U}$ accessible by two rays $\Psi(re^{it_1}),\Psi(re^{it_2})$ which bound a simply connected region containing part of $\partial\mathcal{U}$. This contradicts Lemma \ref{lemboundaryH}.
\end{proof}

\subsection{Global descriptions}\label{sub.sec.global}
 We give a more precise landing property of external rays for parameters on $\partial \mathcal{U}_0$, the connected component of $\mathcal{H}_0$ contained in the right half plan. Notice that it suffices to work in $\overline{S} = \{x+iy;\,x,y\geq 0\}$ by symmetricity of $\mathcal{C}_1$.

\begin{proof}[\textbf{Proof of Theorem~\ref{thm.B}}]
As suggested in the Theorem, we proceed the proof for these three different types of parameters: 
\begin{itemize}
   \item $a_0$ is renormalizable. From the proof of Proposition \ref{prop.mandel} we already know that there are two external rays $\mathcal{R}_{\infty}(\eta),\mathcal{R}_{\infty}(\eta')$ landing at $a_0$ with $\eta\textgreater\eta'$. Moreover these two rays bound a copy of Mandelbrot set $\textbf{M}_{a_0}$. Suppose there is another ray $\mathcal{R}_{\infty}(\zeta)$ landing at $a_0$. Then $\eta\textgreater\zeta\textgreater\eta'$ since from the proof of Proposition \ref{prop.mandel}, $\eta$ (resp. $\eta'$) are decreasing (resp. increasing) limit of $\eta_n$, (resp. $\eta_n'$) and $\mathcal{R}_{\infty}(\eta_n),\mathcal{R}_{\infty}(\eta_n'$ both land at $\partial \mathcal{U}_0$. Suppose $\textbf{M}_{a_0}$ is bounded by $\mathcal{R}_{\infty}(\eta),\mathcal{R}_{\infty}(\zeta)$. Take a rational angle $\zeta'\in(\eta',\zeta)$ such that $R^{\infty}_{a_0}(\zeta')$ lands at some repelling periodic point $x(a_0)$ in $J(f^k_{a_0})$, the Julia set of the renormalizable map. Then $R^{\infty}_{a_0}(\zeta')$ enters every $P^{a_0,v}_{n}$. By Lemma \ref{lem.homeo.para.dym}, the corresponding parameter ray $\mathcal{R}_{\infty}(\zeta')$ enters every $\mathcal{P}_n(a_0)$ and hence converges to some $b\in\textbf{M}_{a_0}$. But $b\not=a_0$, otherwise by Proposition \ref{propland}, $ v_-(a_0) = x(a_0)\in J_{a_0}$. This contradicts the assumption that $\textbf{M}_{a_0}$ is bounded by $\mathcal{R}_{\infty}(\eta),\mathcal{R}_{\infty}(\zeta)$.

    \item $a_0$ is non renormalizable but not Misiurewicz parabolic. By Theorem \ref{thm.dym}, the sequence of dynamical puzzles $P^{a_0,v}_{n}$ shrinks to $v_-(a_0)$. Hence there is some external ray $R^{\infty}_{a_0}(t)$ converging to $v_-(a_0)$ (for example consider the sequence of rays $R^{\infty}_{a_0}(t_n)$ in $P^{a_0,v}_{n}$ such that $t_n$ decreasing to some $t$). By Lemma \ref{lem.homeo.para.dym}, the corresponding parameter ray $\mathcal{R}_{\infty}(t)$ enters every $\mathcal{P}_n(a_0)$ and hence converges to $a_0$. Suppose there is another $\mathcal{R}_{\infty}(t')$ converging to $a_0$. Then again by Lemma \ref{lem.homeo.para.dym}, the corresponding dynamical ray $R^{\infty}_{a_0}(t')$ converges to $v_-(a_0)$. This contradicts the following result:
    \begin{lemma*}[Lemma 6.3 in \cite{Roesch}]
    Given $f_a\in Per_1(1)$, $z\in\partial B_a(0)$. If $c_-(a)\in \partial B_a(0)$, then there is a unique external ray landing at $z$ if and only if $z$ is not in the inverse orbit of $c_-(a)$. 
    \end{lemma*}
   
     \item $a_0$ is Misiurewicz parabolic. By Lemma \ref{lemparaland}, there is some  $\mathcal{R}_{\infty}(t)$ with $3^mt = 1$ landing at $a_0$. In the dynamical plan of $f_{a_0}$, $R^{\infty}_{a_0}(t)$ lands at $v_-(a_0)$. Moreover there we showed that $t$ is unique in $\mathbb{Q}/\mathbb{Z}$. Hence if there is another $\mathcal{R}_{\infty}(t')$ landing at $a_0$, then $t'$ is irrational. To fix the idea, let $t'\textgreater t$. Then $\mathcal{R}_{\infty}(t')$ enters every $\mathcal{Q}^+_n$ Take a sequence of $b_n\in\mathcal{R}_{\infty}(t')$ such that $b_n\in\mathcal{Q}^+_n$. Then $(b_n)$ converges to $a_0$ since $\overline{\mathcal{Q}+_n}$ shrinks to $a_0$. Although for a fixed $n$, we didn't construct the graph $Y^{b_n}_k$ for all depth $k$, it can still be defined by taking preimages of $Y^{b_n}_0$, since $t'$ is irrational and $b_n\in\mathcal{R}_{\infty}(t')$ implies that $c_-(b_n)$ is on some external ray with irrational angle, and hence all the rays in $Y^{b_n}_k$ land. The puzzle pieces $(Q^{b_n}_k)^{\pm}$ sharing part of $R^{\infty}_{b_n}(t)$ as boundary are therefore defined for all $k$. Let $R^{\infty}_{b_n}(t^n_k)$ ($t^n_k\textgreater t$) be the external ray in $\partial (Q^{b_n}_k)^+$ landing on $\partial B^*_{b_n}(0)$. 
     From the proof of Proposition \ref{prop.loc.paramisiur} we have seen that $v_-(b_0)\in (Q^{b_n}_n)^{+}$ and thus $R^{\infty}_{b_n}(t')$ enters $(Q^{b_n}_n)^{+}$, $t^n_n\textless t'\textless t$. The following claim tells us that $t^n_k$ is actually independent of $n$.
     \begin{claim*}
     For any $a_1,a_2$ belonging to the same parameter external ray $\mathcal{R}_{\infty}(\theta)$, there exists a quasiconformal mapping $\varphi$ conjugating $f_{a_1}$ to $f_{a_2}$ and preserving the angle of external rays, i.e. $\varphi(R^{\infty}_{a_1}(\eta)) = R^{\infty}_{a_2}(\eta)$.
     \end{claim*}
     Thus we have $t^1_k\textless t'\textless t$. Since $t$ is in the inverse orbit of 0 by multiplication by3, there is a smallest $l$ such that $f_{b_1}^l(R^{\infty}_{b_1}(t)) = R^{\infty}_{b_1}(0)$ and $Q_{k-l} := f_{b_1}^l((Q^{b_1}_k)^{+})$ is a puzzle piece at 0. Thus we have $f_{b_1}(Q_{k+1-l}) = Q_{k-l}$. This implies that $3\cdot3^lt^1_{k+1} = 3^lt^1_{k}$. Hence $3^lt^1_{k}$ converges to 0 as $k\to\infty$, so $t^1_k$ converges to $t$ as $k\to\infty$, which leads to $t = t'$, a contradiction.
     
\end{itemize}
 \begin{proof}[proof of the Claim]
     Let $\phi^{\infty}_{a_i}:W_i\longrightarrow \hat{\mathbb{C}}\setminus\overline{\mathbb{D}_{r_i}}$ be the maximum extended Böttcher coordinate of $f_{a_i}$, where $W_i$ is a neighborhood of $\infty$ whose boundary contains $c_-(a_i)$ and $\mathbb{D}_{r_i}$ is the disk with radius $r_i\textgreater 1$ centered at 0. Define $\tau(z) = |z|^{s}\cdot z$, where $s = \frac{log(r_2)-log(r_1)}{log(r_1)}$. Define a $f_{a_1}-$invariant Beltrami differential $\nu$ as follows: $\nu$ is defined to be the pull back of $(\tau\circ\phi^{\infty}_{a_1})^*0$ by $f_{a_1}$ on $B_{a_1}(\infty)$ and set to be 0 elsewhere. Integrate $\mu$ with some proper normalisation to get $\varphi$ such that $\Tilde{f} := \varphi\circ{f_{a_1}}\circ{\varphi^{-1}}\in Per_1(1)$. It is easy to check that $\varphi$ preserves angles of external rays and that the equipotential of $\varphi(v_-(a_1))$ is $3r_2$. Hence $\Tilde{f}= f_{a_2}$.
     \end{proof}
This ends the proof of the Theorem.
\end{proof}

\begin{definition}[Wake]
Let $a_0\in (\partial\mathcal{U}_0\cap{\overline{S}})\setminus\{0\}$ be renormalisable. Its \textbf{wake} $\mathcal{W}(a_0)$ is defined to be the open region bounded by the two corresponding landing external rays which contains the Mandelbrot set copy $\textbf{M}_{a_0}$. The wakes $\mathcal{W}(0),-\mathcal{W}(0)$ are defined as in Definition \ref{def.wake0}.
\end{definition}
\begin{definition}[Limb]
Let $a_0\in \partial\mathcal{U}_0\cap{\overline{S}}$. The \textbf{limb} $\mathcal{L}(a_0)$ at $a_0$ is defined to be $\overline{\mathcal{W}(a_0)}\cap\mathcal{C}_1$ if $a_0$ is renormalisable and $a_0\not=0$; to be $(\overline{\mathcal{W}(0)\cup-\mathcal{W}(0)})\cap\mathcal{C}_1$ if $a_0 = 0$; otherwise to be $\{a_0\}$.
\end{definition}

\begin{proof}[\textbf{Proof of Theorem~\ref{thm.C}}]
    By symmetricity it suffices to do the proof for $\Tilde{\mathcal{C}}_1 := \mathcal{C}_1\cap\overline{S}\cap(\mathcal{W}(0))^c$. For any $n\geq0$, consider all the Misiurewicz parabolic parameters of depth $n$ (that is, $f^n_a(v_-(a)) = 0$) on $\partial\mathcal{U}_0\cap\overline{S}$ and the corresponding unique landing external rays. These rays together with $\partial\overline{\mathcal{U}_0}$ separate $\Tilde{\mathcal{C}}_1\setminus\mathcal{U}_0$ in to several open sectors of depth $n$. Take any $b\in\Tilde{\mathcal{C}}_1\cap{(\overline{\mathcal{U}_0}})^c$. Let $\mathcal{S}_n(b)$ be the sector of depth $n$ containing $b$. Let $\mathcal{R}_{\infty}(t_n),\mathcal{R}_{\infty}(t'_n)$ be the two external rays bounding $\mathcal{S}_n(b)$ and $a_n,a'_n$ their landing point repectively. Clearly $a_n,a'_n$ converges to some $a_0\in\partial\mathcal{U}_0$. Then $a_0$ must be renormalisable. If not, then by Theorem \ref{thm.B}, $t_n,t'_n$ converge to the same angle, which implies that $b = a_0$, a contradiction since we take $b\not\in\overline{\mathcal{U}_0}$.
    Hence $a_0$ is renormalisable, by Theorem \ref{thm.B}, $t_n,t'_n$ converge respectively to the angles of the two external rays landing at $a_0$, and hence $b\in\mathcal{W}(a_0)$.
\end{proof}

\bibliographystyle{plain}
\bibliography{references}

\begin{thebibliography}{10}

\bibitem{ASENS_1985_4_18_2_287_0}
Adrien Douady and John~Hamal Hubbard.
\newblock On the dynamics of polynomial-like mappings.
\newblock {\em Annales scientifiques de l'\'Ecole Normale Sup\'erieure}, Ser.
  4, 18(2):287--343, 1985.

\bibitem{Faught}
D.~Faught.
\newblock {\em Local connectivity in a family of cubic polynomials}.
\newblock Ph. D. Thesis of Cornell University, 1992.

\bibitem{DH}
A.~Douady{,}~J. Hubbard.
\newblock {\em Étude dynamique des polynômes complexes}.
\newblock Publ. Math. d’Orsay, 1984.

\bibitem{lomonaco}
L.~Lomonaco.
\newblock Parabolic-like mappings.
\newblock {\em Ergodic Theory and Dynamical Systems}, 35(7):2171–2197, 2015.

\bibitem{Lomonaco2012ParameterSF}
Luciana Luna~Anna Lomonaco.
\newblock Parameter space for families of parabolic-like mappings.
\newblock {\em Advances in Mathematics}, 261:200--219, 2012.

\bibitem{Milnor}
J.~Milnor.
\newblock Cubic polynomial maps with periodic critical orbit, part i.
\newblock {\em Complex Dynamics: Families and Friends}, pages 333--412, 01
  2009.

\bibitem{milnor2011dynamics}
J.~Milnor.
\newblock {\em Dynamics in One Complex Variable. (AM-160): (AM-160) - Third
  Edition}.
\newblock Annals of Mathematics Studies. Princeton University Press, 2011.

\bibitem{Nakane}
S.~Nakane.
\newblock Capture components for cubic polynomials with parabolic fixed points.
\newblock {\em Academic Reports Fac. Eng. Tokio Polytech. Univ. (1)}, 28, 01
  2005.

\bibitem{Pommerenke}
C.~Pommerenke.
\newblock {\em Univalent Functions}.
\newblock Vandenhoeck Ruprecht, Göttingen, 1975.

\bibitem{Roesch1}
P.~Roesch.
\newblock Hyperbolic components of polynomials with a fixed critical point of
  maximal order.
\newblock {\em Annales scientifiques de l'\'Ecole Normale Sup\'erieure}, Ser.
  4, 40(6):901--949, 2007.

\bibitem{Roesch}
P.~Roesch.
\newblock {\em Cubic polynomials with a parabolic point}.
\newblock Ergodic Theory and Dynamical Systems, 30(6), 1843-1867.
  doi:10.1017/S0143385709000820, 2010.

\bibitem{MSS}
R.~Mañé{,} P. Sad{,}~D. Sullivan.
\newblock {\em On the dynamics of rational maps}.
\newblock Annales scientifiques de l'École Normale Supérieure, 193-217, 1983.

\bibitem{Sl}
Z.~Słodkowski.
\newblock {\em Extensions of holomorphic motions}.
\newblock Annali della Scuola Normale Superiore di Pisa,, Classe di Scienze 4e
  série, tome 22, 1995.

\bibitem{Tanlei}
L.~Tan and C.~L. Petersen.
\newblock {\em Branner-Hubbard Motions and attracting dynamics}.
\newblock Dynamics on the Riemann sphere, 45–70, Eur. Math. Soc., Zurich,
  2006.

\bibitem{TaYi}
Tan Lei{,}~Yin Yongcheng.
\newblock {\em Local connectivity of the Julia set for geometrically finite
  rational maps}.
\newblock Science in China (Serie A) 39, 39–47, 1996.

\end{thebibliography}

\end{document}